\documentclass[12pt]{amsart}
\usepackage{graphicx} 
\usepackage{url,hyperref}
\usepackage{amsthm,amssymb}
\usepackage[margin = 1in]{geometry}
\usepackage{enumerate}
\usepackage{BAstyle}

\numberwithin{equation}{section}

\title[An Explicit Tauberian Theorem]{An Explicit Tauberian Theorem taking Averaged Inputs with an Application to Counting Abelian Number Fields}
\author{Brandon Alberts}

\usepackage{amsmath}
\usepackage{amssymb}
\usepackage{tikz}
\usetikzlibrary{matrix}

\begin{document}

\begin{abstract}
    Given a Dirichlet series $L(s) = \sum a_n n^{-s}$, the asymptotic growth rate of $\sum_{n\le X} a_n$ can be determined by a Tauberian theorem. Bounds on the error term are typically controlled by the size of $|L(\sigma+it)|$ for fixed real part $\sigma$. We modify this approach to prove new Tauberian theorems with error terms depending only on the average size of $L(\sigma+it)$ as $t$ varies, and we take care to track explicit dependence on various parameters. This often leads to stronger error bounds, and introduces strong connections between asymptotic counting problems and moments of $L$-functions.

    We provide self-contained statements of Tauberian theorems in anticipation that these results can be used ``out of the box'' to prove new asymptotic expansions. We demonstrate this by proving square root saving error bounds for the number of $C_n$-extensions of $\Q$ of bounded discriminant when $n=3$, $4$, $8$, $16$, or $2p$ for $p$ an odd prime.
\end{abstract}

\maketitle

\tableofcontents

\newpage

\section{Introduction}

Given a function $N:\R_{\ge 0}\to \C$, we denote the corresponding Mellin transform by
\[
    L(s,N) := s\int_1^\infty N(x) x^{-s-1}dx.
\]
This generalizes the construction of a generating Dirichlet series: if $N(X) = \sum_{n\le X} a_n$ it follows that $L(s,N) = \sum a_n n^{-s}$ as long as the sum converges absolutely.

The goal of this paper is to prove a Tauberian theorem that uses slightly more stringent analytic conditions in order to give a better error bound. Tauberian theorems typically utilize a vertical bound of the form $|L(s,N)| \ll_{\sigma} (1+|t|)^{\xi}$ in order to control the size of the error term. We show that all references to pointwise bounds can be removed from such Tauberian theorems, in exchange for bounds on twisted moments of the form
\[
    \int_T^{2T} L(\sigma \pm it,N)Z^{\pm it}dt.
\]
It is typically easier to bound such moments than it is to give pointwise bounds, often leading to stronger error bounds.

The following is our main theorem for a nondecreasing function with inexplicit constants. The explicit dependence of the implied constants on various parameters is given by Theorem \ref{thm:main}.

\begin{theorem}\label{thm:inexplicit_main}
    Let $N:\R_{\ge 0}\to \R$ be a nondecreasing function which satisfies the following:
    \begin{enumerate}[(a)]
        \item $L(s,N)$ converges absolutely on the region ${\rm Re}(s) > \sigma_a$,
        \item $L(s,N)$ has a meromorphic continuation\footnote{We take the standard convention that $f(s)$ being meromorphic on a set $S$ means that $f(s)$ is meromorphic on some open neighborhood of $S$.} to ${\rm Re}(s) \ge \sigma_a-\delta$ with at most finitely many poles in this region, and
        \item For each $\sigma\ge \sigma_a - \delta$, each $T\ge e$ for which $L(s,N)/s$ does not have a pole on the vertical line $[\sigma+iT,\sigma+2iT]$ and for each $Z\ge e/2$
        \begin{align}\label{eq:dyadyic_integral_bound}
            \left|\int_{T}^{2T} L(\sigma \pm it,N)Z^{\pm it}dt\right| \ll_{N,\sigma} \begin{cases}
                 T^{\widetilde{\eta}} & \sigma > \sigma_a - \delta\\
                 T^{\eta}(\log T)^{\beta} & \sigma = \sigma_a - \delta
            \end{cases}
        \end{align}
        for some constants $\widetilde{\eta} > 0$, $\eta > 0$, and $\beta \ge 0$.
    \end{enumerate}
    Then for each $X \ge e$
    \[
        N(X) = \sum_{{\rm Re}(z) \ge \sigma_a - \delta} \underset{s=z}{\rm Res}\left(L(s,N)\frac{X^s}{s}\right) + O\left(X^{\sigma_a-\frac{\delta}{\max\{\eta,1\}}} (\log X)^{\theta}\right),
    \]
    where
    \[
        \theta = \begin{cases}
            0 & \eta < 1\\
            \beta+1 & \eta = 1\\
            (b-1)(1-1/\eta) + \beta/\eta & \eta > 1
        \end{cases}
    \]
    and $b$ is the order to which $s=\sigma_a$ is a pole of $L(s,N)$.

\end{theorem}

To our knowledge, \Cref{thm:inexplicit_main} and the explicit version \Cref{thm:main} are the first general statements of this form that do not require pointwise bounds as input at any stage of the argument. Our proof will largely follow Landau's method of finite differencing \cite{landau1915} as done in \cite{roux2011,alberts2024power}, with some similarities to work in \cite{pierce-turnage-butterbaugh-zaman2025guide}, using \eqref{eq:dyadyic_integral_bound} in leu of pointwise bounds or functional equations.

We particularly highlight the power saving in this error bound and compare this with results coming from pointwise vertical bounds: \cite[Theorem 6.1]{alberts2024power} and \cite[Theorem B]{pierce-turnage-butterbaugh-zaman2025guide} show that if
\[
    L(\sigma + it,N) \ll (1+|t|)^{\xi+\epsilon},
\]
then one can prove an error bound of the form \[
    O\left(X^{\sigma_a - \frac{\delta}{\xi+1} + \epsilon}\right).
\]
In fact, we can now see this as a consequence of Theorem \ref{thm:main} by taking $\widetilde{\eta} = \eta = \xi+1+\epsilon$ and $\beta=0$. We generally expect the average to be smaller and easier to control, and indeed one can often show that $\eta$ may be taken smaller than $\xi+1+\epsilon$ in applications.

\begin{remark}
The actual value of $\widetilde{\eta}$ does not play a role in the argument, we only require some polynomial bound for the growth of \eqref{eq:dyadyic_integral_bound} for all $\sigma \ge \sigma_a - \delta$. The size of the error term is then completely determined by the growth of \eqref{eq:dyadyic_integral_bound} on the leftmost edge $\sigma = \sigma_a-\delta$. Although pointwise bounds are not necessary for the argument, in case pointwise bounds of the form $|L(\sigma+it,F)| \ll (1+|t|)^{\xi+\epsilon}$ do exist it suffices to take $\widetilde{\eta} = \xi+1+\epsilon$. This will be sufficient for many applications involving Dirichlet series, where $\eta$ can then be bounded by other means.
\end{remark}

We discuss a few different use cases for \Cref{thm:inexplicit_main} and \Cref{thm:main}, showing how various existing tools in the literature can directly give the bounds in \eqref{eq:dyadyic_integral_bound}. We also state additional Tauberian theorems for special cases that follow from our main result that highlight different considerations, including issues of resonance (\Cref{thm:resonance_main}) and analogous statements for complex-valued $N$ that improve on existing results in the literature (\Cref{thm:inexplicit_complex_main} and \Cref{thm:resonance_complex_main}).

\subsection{Bounding via Integral Moments}
The trivial bound
\[
    \left|\int_T^{2T} L(\sigma+it,N)Z^{it}dt\right| \le \int_T^{2T} |L(\sigma+it,N)| dt
\]
allows one to use upper bounds for integral moments as ``out of the box'' ingredients for \Cref{thm:inexplicit_main}. See, for example, recent work of Languasco, Lunia, and Moree \cite{languasco-lunia-moree2025counting} improving error bounds for the number of ideals of bounded norm in abelian number fields using bounds for small integral moments of Dirichlet $L$-functions.

We bolster this with a few more examples from the study of number field counting: Let $G$ be an abelian group. Define
\[
    \mathcal{F}_{\Q}(G;X) := \{L/\Q : \Gal(L/\Q)\cong G,\ |\disc(L/\Q)|\le X\}.
\]
The asymptotic growth rate of this function was determined by M\"aki \cite{maki1985} to be
\[
    \#\mathcal{F}_{\Q}(G;X) \sim c(\Q,G) X^{1/a(G)}(\log X)^{b(\Q,G)-1},
\]
where $c(\Q,G) > 0$, $a(G) = |G|(1-\ell^{-1})$ for $\ell$ the smallest prime dividing $G$, and $b(\Q,G) = \frac{|G[\ell]|-1}{\phi(\ell)}$ where $\phi$ is Euler's totient function. This was generalized by Wright \cite{wright1989} to base fields other than $\Q$. Subsequent asymptotic expansions, sometimes including lower order terms, have since been given by \cite{frei-loughran-newton2018,alberts2024power}.

The generating Dirichlet series for such extensions has been expressed as products of Dedekind zeta functions (times an absolutely convergent Euler product) in \cite{maki1985,alberts2024power}. Upper bounds for the integral moments of Dedekind zeta functions follow directly from work of Chandrasekharan and Narasimhan \cite{chandrasekharan1963approximate} on approximate functional equations. Using these results as inputs in \Cref{thm:inexplicit_main}, we can prove several asymptotic expansions with unconditional square root saving error bounds.

\begin{corollary}\label{cor:uncond_square_root_savings}
    The following asymptotic expansions hold unconditionally
    \begin{align*}
        \#\mathcal{F}_{\Q}(C_3;X) =& c_2(C_{3})X^{1/2} +O\left(X^{1/4}(\log X)^2\right),\\
        \#\mathcal{F}_{\Q}(C_4;X) =& c_2(C_{4})X^{1/2} + c_3(C_4)X^{1/3} + O\left(X^{1/4}(\log X)^4\right),\\
        \#\mathcal{F}_{\Q}(C_6;X) =& c_3(C_6)X^{1/3} + c_4(C_6)X^{1/4} + c_5(C_6)X^{1/5} + O\left(X^{1/6}(\log X)^5\right),\\
        \#\mathcal{F}_{\Q}(C_8;X) =& c_4(C_{8})X^{1/4} + c_6(C_8)X^{1/6} + c_7(C_8)X^{1/7} + O\left(X^{1/8}(\log X)^8\right),\\
        \#\mathcal{F}_{\Q}(C_{16};X) =& c_8(C_{16})X^{1/8} + c_{12}(C_{16})X^{1/12} + c_{14}(C_{16})X^{1/14} + c_{15}(C_{16})X^{1/15}\\
        &+ O\left(X^{1/16}(\log X)^{17}\right),
    \end{align*}
    and for each odd prime $p\ge 5$
    \begin{align*}
        \#\mathcal{F}_{\Q}(C_{2p};X) &= c_p(C_{2p})X^{\frac{1}{p}} + c_{2p-2}(C_{2p})X^{\frac{1}{2p-2}} + c_{2p-1}(C_{2p})X^{\frac{1}{2p-1}} + O_{p}\left(X^{\frac{1}{2p}}(\log X)^{3}\right).
    \end{align*}
    The constants $c_d(C_n)$ are given explicitly by finite sums of convergent Euler products times special values of Dedekind zeta functions for cyclotomic fields.

    The nonvanishing of $c_d(C_n)$ is equivalent to the nonvanishing of certain Dedekind zeta functions at certain rational values in the interval $(1/2,1]$, see \Cref{subsec:leading_constants} for details.
\end{corollary}

A square root power saving bound of the form $O(X^{1/2a(G)+\epsilon})$ is known for all abelian groups under GRH, see \cite[Corollary 1.6]{alberts2024power}. Unconditionally, this quality of error term was previously known only for $G=C_2$, by classical counting results for squarefree numbers, while the case $G=C_3$ can be reasonably inferred from results for counting ideals of bounded norm in $\Q(\sqrt{-3})$ with square root savings (one consequence of Landau's original work \cite{landau1915}). Corollary \ref{cor:uncond_square_root_savings} is a significant improvement in this direction as it proves square root savings for infinitely many abelian groups. In the language of \cite[Definition 1.3]{alberts2024power}, \Cref{cor:uncond_square_root_savings} proves that $\theta(\Q,C_n) \le 1/2a(C_n)$ for $n=3,4,6,8,16,2p$.

More generally, determining the integral moments of $L$-functions is itself a well-studied and interesting problem. We refer the reader to some references for a more comprehensive overview of this field \cite{conrey2001functions,soundararajan2022distribution}. Many of these conjectures and results are of a similar shape, for example the conjectural moments of the Riemann zeta function:
\[
    \int_T^{2T} |\zeta(1/2+it)|^{2m} dt \sim c_m T (\log T)^{m^2}
\]
for each $m \ge 0$, where $c_m$ is some constant (given explicitly by a conjecture of Keating and Snaith \cite{keating-snaith2000random}) . In the context of Theorem \ref{thm:inexplicit_main} this corresponds to $\sigma_a = 1$, $\delta = 1/2$, $\eta = 1$, and $\beta = m^2$, which is enough to produce square root savings.

The Lindel\"of hypothesis is equivalent to an upper bound $\ll_{m,\epsilon} T^{1+\epsilon}$ for these moments for all $m\ge 0$, although bounds for a \emph{particular} value of $m$ are significantly weaker. In the case of the Riemann zeta function sharp upper bounds are known for $0 \le m\le 2$ \cite{heap-radziwill-sound2019sharp}, while they are known for all $m\ge 0$ under RH \cite{harper2013sharp}. This makes Theorem \ref{thm:main} a promising tool for improving error bounds, especially for ``low degree'' functions $L(s,N)$ where we have a chance of unconditionally bounding the corresponding moments.

\subsection{Bounding via (Approximate) Functional Equations}

Suggested by the use of the general approximate functional equation of \cite{chandrasekharan1963approximate} above, \eqref{eq:dyadyic_integral_bound} can also be bounded directly in terms of an approximate functional equation by simply integrating the approximate functional equation itself.

In a similar vein, a functional equation can be used to bound \eqref{eq:dyadyic_integral_bound} when $\sigma < 0$ (or otherwise on the ``other side'' of the functional equation). This process is similar to that used by Landau \cite{landau1915} and following work \cite{chandrasekharan1962functional,lowry-duda-taniguchi-thorne2022uniform} to bound the error term by taking advantage of a functional equation with a particular shape. Bounding \eqref{eq:dyadyic_integral_bound} for $\sigma < 0$ is more-or-less equivalent to bounding $W_k(X)$ in \cite[Equation (18)]{lowry-duda-taniguchi-thorne2022uniform}.

In a sense, our results can be understood as explicitly separating Landau's method into a part that does not require a functional equation, \Cref{thm:inexplicit_main}, and a part that may be bounded using a function equation as in Landau's original approach, namely equation \eqref{eq:dyadyic_integral_bound}. One benefit of this separation is the possibility of improved error terms by bounding \eqref{eq:dyadyic_integral_bound} through other means. For counting ideals in abelian number fields, the work of \cite{languasco-lunia-moree2025counting} produces smaller error bounds than Landau's original method \cite{landau1915,chandrasekharan1962functional,lowry-duda-taniguchi-thorne2022uniform} by appealing to integral moments of Dirichlet $L$-functions instead of the functional equation.

Another benefit of this separation is that one can now modify Landau's use of the functional equation to bound \eqref{eq:dyadyic_integral_bound} using functional equations of new shapes without having to rework through the rest of the finite differencing method.

\subsection{Avoiding certain resonance}

The factor $Z^{it}$ suggests further cancellation, beyond what we could get using the integral moments as upper bounds. It is possible for $L(\sigma+it,N)$ and $Z^{it}$ to resonate with each other for certain $Z$ values, causing \eqref{eq:dyadyic_integral_bound} to be larger than expected. For example, if $L(s,N)$ is a Dirichlet series, the terms $n^{-\sigma-it}$ resonates with $n^{it}$. This leads the case $Z^{it}=n^{it}$ to conspire with $L(\sigma+it,N)$ and may cause the integral in \eqref{eq:dyadyic_integral_bound} to be larger than for the typical choice of $Z$. For Dirichlet series, this issue is often avoided by showing that it suffices to restrict to half-integers.

We will see in \Cref{thm:main} that very few values of $Z$ are needed in \eqref{eq:dyadyic_integral_bound} for our arguments to work. By keeping track of this restriction we can force $Z$ to avoid ``bad'' resonance, at the expense of an extra error term that depends on the frequency of allowed values for $Z$.

\begin{theorem}\label{thm:resonance_main}
    Let $S\subseteq \R_{\ge 1}$ be a discrete, unbounded set of real numbers containing $1$ such that for each $\lambda\in S$,
    \[
        \min\{|\lambda-r| : r\in S\} \ll_S \lambda^c
    \]
    for some fixed $c < 1$. Let $N:\R_{\ge 0}\to \R$ be as in \Cref{thm:inexplicit_main}, except that  for part (c) we only ask that \eqref{eq:dyadyic_integral_bound} holds for $Z\in S$ with $Z\ge e/2$.
    
    Then for each $X\in S$ with $X \gg_S e$ sufficiently large in terms of $S$,
    \[
        N(X) = \sum_{{\rm Re}(z) \ge \sigma_a - \delta} \underset{s=z}{\rm Res}\left(L(s,N)\frac{X^s}{s}\right) + O\left(X^{\sigma_a-\frac{\delta}{\max\{\eta,1\}}} (\log X)^{\theta}\right) + O \left(X^{\sigma_a+c-1}(\log X)^{b-1}\right),
    \]
    where $\theta$ and $b$ are as in \Cref{thm:inexplicit_main}.
\end{theorem}

Consider the case that $S$ is the set of half-integers (together with $1$). These are equally spaced, so we can take $c=0$ which restricts the power savings to at most $O(X^{\sigma_a-1+\epsilon})$. This is reasonable, as when counting the set of half-integers themselves, the growth rate is $X+O(1)$ with a sharp error term.


\subsection{Complex-valued functions}

The Tauberian theorem \cite[Theorem 6.1]{alberts2024power}, following work of Roux \cite{roux2011}, applies to Dirichlet series with arbitrary complex coefficients, as long as a suitable Dirichlet series with nonegative coefficients can be used for comparison. The mechanism by which this comparison works extends to arbitrary functions $N:\R_{\ge 0} \to \C$ and is compatible with our methods.

\begin{theorem}\label{thm:inexplicit_complex_main}
    Let $N:\R_{\ge 0} \to \C$ be a complex function and $\widehat{N}:\R_{\ge 0}\to \R$ be a nondecreasing function which satisfies $|N(X)-N(Y)|\le \widehat{N}(X)-\widehat{N}(Y)$ for each $X\ge Y \ge 0$, and suppose that both $N,\widehat{N}$ satisfy the hypotheses in \Cref{thm:inexplicit_main}(a,b,c) with the same parameters $\sigma_a$, $\delta$, $\widetilde{\eta}$, $\eta$, and $\beta$. Then for each $X\ge e$
    \[
        N(X) = \sum_{{\rm Re}(z) \ge \sigma_a - \delta} \underset{s=z}{\rm Res}\left(L(s,N)\frac{X^s}{s}\right) + O\left(X^{\sigma_a-\frac{\delta}{\max\{\eta,1\}}} (\log X)^{\theta}\right),
    \]
    where $\theta$ is defined as in \Cref{thm:inexplicit_main} with $b$ the largest order of a pole of $L(s,N)$ or $L(s,\widehat{N})$ on the line ${\rm Re}(s) = \sigma_a-\delta$.
\end{theorem}

In the case that $N(X) = \sum_{n\le X} a_n$ and $\widehat{N}(X) = \sum_{n\le X}\widehat{a}_n$ are summatory functions, the comparison $|N(X)-N(Y)|\le\widehat{N}(X)-\widehat{N}(Y)$ is equivalent to $|a_n|\le\widehat{a}_n$, as
\[
    |a_n| = |N(n+1/2) - N(n-1/2)| \le \widehat{N}(n+1/2)-\widehat{N}(n-1/2) = \widehat{a}_n.
\]

\begin{remark}
    We will see in the main result \Cref{thm:main} that our methods also apply when the analytic parameters for $L(s,N)$ differ from the analytic parameters for $L(s,\widehat{N})$. While trickier to optimize at this level of generality, this can be useful in particular if the bounds in \eqref{eq:dyadyic_integral_bound} differ for $N$ and $\widehat{N}$.
\end{remark}

This comparison is consistent with avoiding resonance, so that we can also conclude a complex-valued analog to \Cref{thm:resonance_main}.

\begin{theorem}\label{thm:resonance_complex_main}
    Let $S,c$ be as in \Cref{thm:resonance_main} and let $N,\widehat{N}:\R_{\ge 0}\to \C$ be as in \Cref{thm:inexplicit_complex_main}, except that for hypothesis \Cref{thm:inexplicit_main}(c) we only ask that \eqref{eq:dyadyic_integral_bound} holds for $Z\in S$ with $Z\ge e/2$.
    
    Then for each $X\in S$ with $X \gg_S e$ sufficiently large in terms of $S$,
    \begin{align*}
        N(X) = \sum_{{\rm Re}(z) \ge \sigma_a - \delta} \underset{s=z}{\rm Res}\left(L(s,N)\frac{X^s}{s}\right) &+ O\left(X^{\sigma_a-\frac{\delta}{\max\{\eta,1\}}} (\log X)^{\theta}\right) + O\left(X^{\sigma_a + c-1}(\log X)^{b-1}\right),
    \end{align*}
    where $\theta$ is as in \Cref{thm:inexplicit_complex_main}.
\end{theorem}

\section*{Acknowledgements}

The author thanks Alina Bucur, Kevin McGown, and Amanda Tucker for initial conversations towards the results of this paper.

\section{An Explicit Tauberian Theorem}

We state our main results in this section. We model our notation on that of \cite{chandrasekharan1962functional,lowry-duda-taniguchi-thorne2022uniform}, which allows us both to directly indicate where our approach differs and to more easily take advantage of some basic lemmas. We note that, in some respects, our method has more in common with other works, such as \cite{roux2011,alberts2024power}, because we do not require a functional equation as input.

\subsection{Notation and Assumptions}\label{subsec:hypotheses}

We set the following notations and assumptions for a function $N:\R_{\ge 0} \to \C$ together with positive real numbers $X,T > 0$.
\begin{itemize}
    \item (Meromorphic continuation) $L(s,N)$ converges absolutely for ${\rm Re}(s) > \sigma_a$, has a holomorphic continuation to ${\rm Re}(s) > \sigma_p$, and has a meromorphic continuation to the right half-plane ${\rm Re}(s) \ge \sigma_p - \delta$ with at most finitely many poles in this region.

    Let $T_0 \ge e$ be a number such that all poles in this region have their imaginary part contained in the interval $[-T_0,T_0]$.

    \item (Polar data) Define
    \[
        S_{N,\alpha}^0(X) := \frac{1}{2\pi i} \int_{C_{\alpha}} L(s,N)\frac{X^s}{s} ds = \sum_{{\rm Re}(z) \ge \alpha} \underset{s=z}{\rm Res}\left(L(s,N) \frac{X^s}{s}\right),
    \]
    where $C_{\alpha}$ is any contour contained in the region of meromorphicity which encircles all the poles of $L(s,N)/s$ satisfying ${\rm Re}(s) \ge \alpha$.

    By direct computation of the residue, this function has the form $S_{N,\alpha}^0(X) = \sum X^{z} R_z(\log X)$ for certain polynomials $R_z$. We further define
    \[
        R_{N,\alpha}(X) := \sum X^{{\rm Re}(z)}R_z^{\rm abs}(\log X),
    \]
    where $R_z^{\rm abs}$ is the polynomial obtained from $R_z$ by taking the absolute value of the coefficients.


    \item (Polynomial bounds for twisted moments) We ask that there exists $\widetilde{\eta} \in \R$ such that
    \[
        \left\lvert\int_{T_0}^{T_1} L(\sigma \pm it,N) X^{\pm it} dt\right\rvert \ll T_1^{\widetilde{\eta}}
    \]
    for each $T_1\ge T_0$ and each $\sigma_p - \delta \le \sigma \le \sigma_a + 1$, where the implied constant is independent of $T_1$ (but otherwise may depend on any other parameters). 
    
    \item (Explicit bound for the leftmost twisted moment) We also ask that there are parameters $Q > 0$ and $\eta,\beta\in \R$ such that for each $Z\in \{X, X(1-3T^{-1})\}$
    \begin{align}\label{eq:hypothesis_twisted_moment}
        \left\lvert\int_{T_0}^{T_1} L(\sigma_p-\delta \pm it,N) Z^{\pm it} dt\right\rvert \le QT_1^{\eta}(\log T_1)^{\beta}
    \end{align}
    for each $T_1 \ge T_0$. One generally expects $\eta \ge \widetilde{\eta}$ to follow from the Phragm\'en-Lindel\"of principle, although this is not required.
\end{itemize}

We utilize the standard notation for contour integrals along the vertical line ${\rm Re}(s) = \alpha$,
\[
    \int_{(\alpha)} := \lim_{y\to \infty} \int_{\alpha-iy}^{\alpha+iy}.
\]

\subsection{The main theorem with explicit constants}

The following is our main theorem. We leave one floating parameter $T$ that can be used to optimize the statement. \Cref{thm:inexplicit_main} is proven by choosing $T$ to optimize the $X$ dependence, but we leave $T$ unspecified in the general statement. As suggested by the notation, we will break up integrals of the form $\int_{(\alpha)}$ at the imaginary parts $\pm T$ during the course of our proof.

\begin{theorem}\label{thm:main}
    Let $N:\R_{\ge 0}\to\R$ be a nondecreasing integrable function, $X\ge e$, and $T\ge 6$ satisfying the assumptions in \Cref{subsec:hypotheses}. Then the bound
    \[
        \left\lvert N(X) - S_{N,\sigma_p - \delta}^0(X)\right\rvert \ll E_{N,1}(X,T) + E_{N,2}X^{\sigma_p-\delta}
    \]
    holds, where the implied constant depends only on $\sigma_p$, and on the locations and orders of the poles of $L(s,N)/s$ in the region ${\rm Re}(s) \ge \sigma_p-\delta$.

    The parameters are defined as follows:
    \begin{enumerate}[(a)]        
        \item $\gamma$ is any contour from $\sigma_p-\delta-iT_0$ to $\sigma_p-\delta+iT_0$ which is within the region of meromorphicity, within the half-plane ${\rm Re}(s) \le \sigma_p-\delta$, for which $L(s,N)/s$ has no poles on $\gamma$, and for which all poles of $L(s,N)/s$ between $\gamma$ and the line ${\rm Re}(s) = \sigma_p-\delta$ actually lie on the line ${\rm Re}(s) = \sigma_p-\delta$.
        
        \item the parameter $E_{N,1}(X)$ is given explicitly by:
        \begin{align*}
        E_{N,1}(X,T) := \begin{cases}
            \displaystyle\frac{R_{N,\sigma_p-\delta}(X)}{T} + \left(\frac{k^k}{\log T} + 2^{\delta+1}Q\right)(\log T)^{\beta+1}X^{\sigma_p-\delta} & \eta = 1\\
            \displaystyle \frac{R_{N,\sigma_p-\delta}(X)}{T} + \left(k^k + \frac{2^{\delta}\eta}{|\eta-1|}Q\right)T^{\eta-1}\left(\log T\right)^{\beta}X^{\sigma_p-\delta}& \eta \ne 1,
        \end{cases}
        \end{align*}
        where $k = \lceil \max\{2,\widetilde{\eta}-2,3\eta-3\}\rceil$. 

        \item the parameter $E_{N,2}$ is given explicitly by:
        \begin{align*}
        E_{N,2} := \begin{cases}
            \displaystyle 2^{\delta}Q + 2^{\delta}\int_\gamma\left\lvert\frac{L(s,N)}{s}\right\rvert ds & \eta = 1,\\
            \displaystyle \frac{2^{\delta}e^{\eta-1}}{|\eta-1|}Q + 2^{\delta}\int_\gamma\left\lvert\frac{L(s,N)}{s}\right\rvert ds & \eta \ne 1.
        \end{cases}
        \end{align*}
    \end{enumerate}
    
    Moreover, if $N:\R_{\ge 0}\to \C$ is any function satisfies the assumptions in \Cref{subsec:hypotheses} and $\widehat{N}:\R_{\ge 0}\to \R$ is a nondecreasing function satisfying the same assumptions with parameters $\widehat{\sigma}_a$, $\widehat{\sigma}_p$, $\widehat{\delta}$, $\widetilde{\eta}$, $\widehat{Q}$, $\widehat{\eta}$, $\widehat{\beta}$ satisfying $|N(X)-N(Y)|\le \widehat{N}(X)-\widehat{N}(Y)$ for each $X\ge Y \ge 0$, then for each $X \ge e$ and $T \ge 6$ the bound
    \begin{align*}
        \left\lvert N(X) - S_{N,\sigma_p - \delta}^0(X)\right\rvert \ll & E_{N,1}(X,T) + E_{N,2}X^{\sigma_p-\delta}\\
        &+ E_{\widehat{N},1}(X,T) + E_{\widehat{N},2}X^{\widehat{\sigma}_p-\widehat{\delta}}
    \end{align*}
    holds with $k = \lceil \max\{2,\widetilde{\eta}-2,3\eta-3,3\widehat{\eta}-3\}\rceil$, where the implied constant depends only on $\sigma_p$, $\widehat{\sigma}_p$, and on the locations and orders of the poles of $L(s,N)/s$ in the region ${\rm Re}(s) \ge \sigma_p-\delta$ and of $L(s,\widehat{N})/s$ in the region ${\rm Re}(s) \ge \widehat{\sigma}_p-\widehat{\delta}$.
\end{theorem}

We note that $\gamma$ can be chosen to ``complete the vertical line'' if there are no poles with real part $\sigma_p - \delta$. While this is often the case, we anticipate some applications where allowing a pole on this line will be important.

One directly verifies that $E_{N,1}(X,X^{\delta/\max\{\eta,1\}})$ grows at most like $X^{\sigma_p-\frac{\delta}{\max\{\eta,1\}}}$ times a power of $\log X$, showcasing the optimum power of $X$ we are able to achieve. When proving \Cref{thm:inexplicit_main}, we will see how to choose $T$ to optimize the power of $X$ and $\log X$. Choosing $T$ to optimize the value of $E_{N,1}(X,T)$ with respect to the dependence of the error term on $Q$, $\delta$, $\eta$, $\widetilde{\eta}$ and the coefficients of $R_{N,\sigma_p-\delta}(X)$ is certainly possible, however we have found that attempting to do so at this level of generality results in very complicated expressions that are less than enlightening. It may be more feasible to evaluate $R_{N,\sigma_p-\delta}(X)$ for a specific application first, then choose $T$ to optimize the dependence on all these parameters.

We demonstrate how to choose an optimum value for $T$ by proving \Cref{thm:inexplicit_main} and \Cref{thm:inexplicit_complex_main}.

\begin{proof}[Proof of \Cref{thm:inexplicit_main} and \Cref{thm:inexplicit_complex_main}]
    Let $\sigma_p=\sigma_a$. We bound the term
    \[
        E_{N,1}(X,T) + E_{N,2}X^{\sigma_a-\delta}
    \]
    by making optimum choices for $T$. When proving \Cref{thm:inexplicit_complex_main}, we remark that when  we take $\widehat{\sigma}_p=\widehat{\sigma}_a$ the quantity
    \[
        E_{\widehat{N},1}(X,T) + E_{\widehat{N},2}X^{\widehat{\sigma}_p-\widehat{\delta}}
    \]
    will have the same bound, as our optimum choice for $T$ will depend only on the parameters $\delta=\widehat{\delta}$, $\eta=\widehat{\eta}$, $\beta=\widehat{\beta}$, and $b=\widehat{b}$. Thus, \Cref{thm:inexplicit_complex_main} follows from the same argument as \Cref{thm:inexplicit_main}.

    We note that by construction $R_{N,\sigma_p-\delta}(X) \ll X^{\sigma_p}(\log X)^{b-1} \le X^{\sigma_a}(\log X)^{b-1}$, which we use throughout this argument. If $\eta < 1$, sending $T\to \infty$ implies $E_{N,1}(X,T)\to 0$. Thus taking $T$ sufficiently large implies
    \[
        E_{N,1}(X,T) + E_{N,2}X^{\sigma_a-\delta} \ll X^{\sigma_a-\delta}.
    \]
    The implied constant is precisely $E_{N,2}$, but this is not necessary for the statement of \Cref{thm:inexplicit_main}. For the remaining cases, we let the implied constant absorb every parameter except for $X$ and $T$.
    
    If $\eta = 1$, the error term is bounded above by
    \[
        E_{N,1}(X,T) + E_{N,2}X^{\sigma_a-\delta} \ll T^{-1}X^{\sigma_a}(\log X)^{b-1} + (\log T)^{\beta+1}X^{\sigma_a-\delta} + X^{\sigma_a-\delta}.
    \]
    The first term is decreasing in $T$ while the second is increasing, so the optimum value is when the two terms are equal. Taking $T = X^{\delta}(\log X)^{b-\beta-2}$ makes the two terms have the same order of magnitude, agreeing with the claimed order of magnitude in \Cref{thm:inexplicit_main}.

    If $\eta > 1$, the error term is bounded above by
    \[
        E_{N,1}(X,T) + E_{N,2}X^{\sigma_a-\delta} \ll T^{-1}X^{\sigma_a}(\log X)^{b-1} + T^{\eta-1}(\log T)^{\beta}X^{\sigma_a-\delta} + X^{\sigma_a-\delta}.
    \]
    Similar to the $\eta = 1$ case, this is optimized when $T = X^{\delta/\eta}(\log X)^{(b-1-\beta)/\eta}$. Plugging this in and evaluating concludes the proof.
\end{proof}

\subsection{The explicit main theorem avoiding certain resonance}

We can likewise state a completely explicit version of \Cref{thm:resonance_main}.

\begin{theorem}\label{thm:explicit_resonance}
    Let $c_1 > 0$, $c_2 < 1$, and let $S\subseteq \R_{\ge 1}$ be a discrete, unbounded set of real numbers containing $1$ such that for each $\lambda\in S$,
    \[
        \min\{|\lambda-r| : r\in S\} \le c_1\lambda^{c_2}.
    \]
    Let $N,\widehat{N}:\R_{\ge 0}\to \C$ be as in \Cref{thm:main}, except we only ask that \eqref{eq:hypothesis_twisted_moment} holds for $Z\in S$ with $e/2\le Z \le X$. Then for each $X \in S$ with $X\ge e(3c_1)^{\frac{1}{1-c_2}}$ and each $T\ge 12$, the bound
    \begin{align*}
        \left\lvert N(X) - S_{N,\sigma_p - \delta}^0(X)\right\rvert \ll& E_{N,1}(X,T) + E_{N,2}X^{\sigma_p-\delta} + E_SX^{c_2-1}R_{N,\sigma_p-\delta}(X) + E_{N,S,3}(X,T)\\
        & + E_{\widehat{N},1}(X,T) + E_{\widehat{N},2}X^{\widehat{\sigma}_p-\widehat{\delta}} + E_SX^{c_2-1}R_{\widehat{N},\widehat{\sigma}_p-\widehat{\delta}}(X) + E_{\widehat{N},S,3}(X,T)
    \end{align*}
    holds, the implied constant depends only on $\sigma_p$, $\widehat{\sigma}_p$, and on the locations and orders of the poles of $L(s,N)/s$ in the region ${\rm Re}(s) \ge \sigma_p-\delta$ and of $L(s,\widehat{N})/s$ in the region ${\rm Re}(s) \ge \widehat{\sigma}_p-\widehat{\delta}$.

    The parameters are defined by \Cref{thm:main}, along with $E_S = c_1(1 + (3/4)^{c_2})$ and
    \[
        E_{N,S,3}(X,T) = \begin{cases}
            \displaystyle \left(k^k + \frac{2^{\delta}\eta}{1-\eta}Q\right)E_S^{1-\eta}(\log T)^{\beta}X^{\sigma_p-\delta + (c_2-1)(1-\eta)} & \eta < 1\\
            0 & \eta \ge 1,
        \end{cases}
    \]
    with the analogous definition for $E_{\widehat{N},S,3}(X,T)$.
\end{theorem}

The proofs of \Cref{thm:resonance_main} and \Cref{thm:resonance_complex_main} are then immediate after noting that $E_{N,S,3}(X,X^a) \ll_a X^{\sigma_p-\delta}$, following the same argument as in the proofs of \Cref{thm:inexplicit_main} and \Cref{thm:inexplicit_complex_main}.

\begin{proof}[Proof of \Cref{thm:explicit_resonance}]
    We define the $S$-floor of a real number $x\ge 1$, denoted $\lfloor x\rfloor_S$, to be the largest $\lambda\in S$ with $\lambda \le x$. Likewise we define the $S$-ceiling of $x\ge 1$, denoted $\lceil x\rceil_S$, to be the smallest $\lambda\in S$ with $\lambda \ge x$. These exist by $S$ discrete, unbounded, and containing $1$. For any real number $x\ge 1$, we note that
    \begin{align*}
        x - \lfloor x\rfloor_S \le \lceil x\rceil_S - \lfloor x\rfloor_S \le \min\{c_1\lceil x\rceil_S^{c_2}, c_1\lfloor x\rfloor_S^{c_2}\} \le c_1 x^{c_2},
    \end{align*}
    where the last bound comes from $\lceil x\rceil_S\ge x$ when $c_2 < 0$, and from $\lfloor x\rfloor_S\le x$ when $c_2\ge 0$.

    For any real $T$, we choose $T_S$ to be such that
    \[
        X(1-3T_S^{-1}) = \lfloor X(1-3T^{-1})\rfloor_S \in S.
    \]
    We will apply \Cref{thm:main} using $T_S$, for which we need $T_S\ge 6$ or equivalently $\lfloor X(1-3T^{-1})\rfloor_S \ge \frac{1}{2}X$. We know that
    \begin{align*}
        \lfloor X(1-3T^{-1})\rfloor_S &\ge X(1-3T^{-1}) - c_1(X(1-3T^{-1}))^{c_2}\\
        &\ge X(1-3T^{-1})\left(1 - c_1(X(1-3T^{-1}))^{c_2-1}\right)\\
        &\ge \frac{3}{4}X\left(1 - c_1\left(\frac{3}{4}X\right)^{c_2-1}\right),
    \end{align*}
    where the last line follows from $T\ge 12$ and $c_2\le 1$. Finally, the bound $X \ge e(3c_1)^{\frac{1}{1-c_2}}$ implies
    \begin{align*}
        \lfloor X(1-3T^{-1})\rfloor_S &\ge \frac{3}{4}X\left(1 - c_1\left(\frac{3e}{4}\right)^{c_2-1}(3c_1)^{-1}\right) \ge \frac{3(1-3^{-1})}{4}X = \frac{1}{2}X.
    \end{align*}

    Solving for $T_S$, we get
    \[
        T_S = \frac{3X}{X - \lfloor X(1-3T^{-1})\rfloor_S}.
    \]
    It is immediate that $T_S\le T$, meanwhile
    \begin{align*}
        T_S^{-1} &\le \frac{X - (X(1-3T^{-1})) + c_1(X(1-3T^{-1}))^{c_2}}{3X}\\
        &= T^{-1} + \frac{1}{3}c_1X^{c_2-1}(1-3T^{-1})^{c_2}\\
        &\ll T^{-1} + c_1\max\{1,(3/4)^{c_2}\}X^{c_2-1}\\
        &\ll T^{-1} + E_S X^{c_2-1}
    \end{align*}
    follows from $T\ge 12$, where the implied constant is absolute.

    For $\eta < 1$, we can choose any value of $T\ge 12$ and apply \Cref{thm:main} using $T_S\ge 6$ and $X\ge e(3c_1)^{\frac{1}{1-c_2}}\ge e$ to bound
    \begin{align*}
        E_{N,1}(X,T_S)=&T_S^{-1}R_{N,\sigma_p-\delta}(X) + \left(k^k + \frac{2^{\delta}\eta}{1-\eta}Q\right)T_S^{\eta-1}(\log T_S)^{\beta}X^{\sigma_p-\delta}\\
        \ll& T^{-1}R_{N,\sigma_p-\delta}(X) + E_SX^{c_2 - 1}R_{N,\sigma_p-\delta}(X)\\
        &+ \left(k^k + \frac{2^{\delta}\eta}{1-\eta}Q\right)(T^{-1} + E_SX^{c_2-1})^{1-\eta}(\log T)^{\beta}X^{\sigma_p-\delta}
    \end{align*}
    Consider that if $x,y>0$ with $x+y=1$, then $0\le \eta \le 1$ implies $x^{1-\eta} + y^{1-\eta} \ge x+y = 1$. Setting $x=\frac{T^{-1}}{T^{-1}+E_SX^{c_2-1}}$ and $y=\frac{E_S X^{c_2-1}}{T^{-1}+E_S X^{c_2-1}}$ and clearing denominators implies
    \begin{align*}
        E_{N,1}(X,T_S) \ll& T^{-1}R_{N,\sigma_p-\delta}(X) + E_SX^{c_2 - 1}R_{N,\sigma_p-\delta}(X)\\
        &+ \left(k^k + \frac{2^{\delta}\eta}{1-\eta}Q\right)(T^{\eta-1} + E_S^{1-\eta}X^{(c_2-1)(1-\eta)})(\log T)^{\beta}X^{\sigma_p-\delta}\\
        &= E_{N,1}(X,T) + E_SX^{c_2 - 1}R_{N,\sigma_p-\delta}(X) + E_{N,S,3}(X,T).
    \end{align*}
    
    Doing the same for $\eta=1$ with the bound $T_S\le T$ gives
    \begin{align*}
        E_{N,1}(X,T_S)&=T_S^{-1}R_{N,\sigma_p-\delta}(X) + \left(k^k(\log T_S)^{\beta} + 2^{\delta+1}Q(\log T_S)^{\beta+1}\right)X^{\sigma_p-\delta}\\
        &\ll T^{-1}R_{N,\sigma_p-\delta}(X) + E_S X^{c_2 - 1}R_{N,\sigma_p-\delta}(X) + \left(k^k(\log T)^{\beta} + 2^{\delta+1}Q(\log T)^{\beta+1}\right)X^{\sigma_p-\delta}\\
        &= E_{N,1}(X,T)+ E_S X^{c_2 - 1}R_{N,\sigma_p-\delta}(X),
    \end{align*}
    and $\eta > 1$ gives
    \begin{align*}
        E_{N,1}(X,T_S)&=T_S^{-1}R_{N,\sigma_p-\delta}(X) + \left(k^k + \frac{2^{\delta}\eta}{\eta-1}Q\right)T_S^{\eta-1}(\log T_S)^{\beta}X^{\sigma_a-\delta}\\
        &\ll T^{-1}R_{N,\sigma_p-\delta}(X) + E_S X^{c_2 - 1}R_{N,\sigma_p-\delta}(X) + \left(k^k + \frac{2^{\delta}\eta}{\eta-1}Q\right)T^{\eta-1}(\log T)^{\beta}X^{\sigma_a-\delta}\\
        &= E_{N,1}(X,T)+ E_S X^{c_2 - 1}R_{N,\sigma_p-\delta}(X).
    \end{align*}
    Incorporating these bounds in the error in \Cref{thm:main} gives the first row of the error for \Cref{thm:explicit_resonance}. The quantity $E_{\widehat{N},1}(X,T_S)$ is bounded in exactly the same way, giving the second row of the error term.
\end{proof}

\subsection{Summary of the proof of \Cref{thm:main}}
We will largely follow Landau's method of finite differencing. The process proceeds by smoothing $N(X)$ using finite difference operators, shifting the contour in Perron's formula (\Cref{sec:contour_shift}), then unsmoothing back to $N(X)$ (\Cref{sec:unsmoothing}). This is similar to Landau's original work \cite{landau1915} and the following works \cite{chandrasekharan1962functional,lowry-duda-taniguchi-thorne2022uniform,roux2011,alberts2024power}.

In previous examples of this method, pieces of the shifted contour are generally bounded using pointwise bounds for $|L(s,N)|$. Integration by parts can be used in some places to apply average bounds instead, but in order to avoid the need for pointwise bounds at any point in the argument at all we will need to use an auxiliary average
\[
    N(X) = \frac{1}{T} \int_T^{2T} N(X) du.
\]
This equality is trivial, as the variable of integration does not actually appear inside the integral. The way in which we shift the contour integral in Perron's formula will then be chosen to depend on the variable of integration $u$. In this way, we are able to take an average of several different choices for shifting the contour.

\section{Contour shifting and a smoothed counting function}\label{sec:contour_shift}

Let $k$ be a positive integer. We will first consider the smoothed functions
\[
    A^k_N(X) := \int_0^X \int_0^{x_1}\cdots \int_0^{x_{k-1}} N(x_k) dx_{k}\cdots dx_1
\]
for a suitably large value of $k$, where $A_N^0(X) = N(X)$ by definition.

\begin{lemma}\label{lem:contour_shift}
    Let $N$ be a function which satisfies the assumptions in \Cref{subsec:hypotheses}, $k \ge \max\{2,\widetilde{\eta} - 2\}$ an integer, and $X > 0$. Then
    \[
        A^k_N(X) = \frac{1}{2\pi i} \left(\int_C + \int_{(\sigma_p-\delta)}^*\right)L(s,N)\frac{\Gamma(s)}{\Gamma(s+k+1)} X^{s+k} ds,
    \]
    where $C$ is a closed contour around the poles of the integrand in the region ${\rm Re}(s) \ge \sigma_p - \delta$ and
    \[
        \int_{(\sigma_p-\delta)}^* = \int_{\sigma_p-\delta-i\infty}^{\sigma_p-\delta-iT_0} + \int_{\sigma_p-\delta+iT_0}^{\sigma_p-\delta+i\infty} + \int_\gamma
    \]
    where $T_0 > 0$ is larger than the imaginary part of any pole of $L(s,N)$ on the line ${\rm Re}(s) = \sigma_p - \delta$ and $\gamma$ a contour from $\sigma_p-\delta-i T_0$ to $\sigma_p-\delta +iT$ contained in the region of meromorphicity of $L(s,N)\frac{\Gamma(s)}{\Gamma(s+k+1)}$, lying to the left of ${\rm Re}(s) = \sigma_p-\delta$ and for which there are no poles on $\gamma$ or strictly between $\gamma$ and ${\rm Re}(s) = \sigma_p - \delta$.
\end{lemma}

We note that if there are no poles with real part $\sigma_p - \delta$, $\gamma$ can be chosen to ``complete the vertical line'' making $\int_{(\sigma_p-\delta)}^* = \int_{(\sigma_p - \delta)}$.

\begin{remark}
    This is closely related to \cite[Lemma B.4]{pierce-turnage-butterbaugh-zaman2025guide}, relying on the notion that for a sufficiently smoothed counting function we expect that the contour integral in Perron's formula can be shifted to the left without any error terms. The primary difference between Lemma \ref{lem:contour_shift} and previous work on smoothed counting functions is that we only assume the existence of a polynomial bound for the twisted average
    \[
        \int_{T_0}^{T} L(\sigma\pm it,N)X^{\pm it} dt \ll T^{\widetilde{\eta}},
    \]
    while all previous work the author is aware of requires polynomial pointwise bounds of the form
    \[
        |L(\sigma+it,N)| \ll T^{\xi}.
    \]
    Speaking theoretically this improvement is mildly interesting, although perhaps not surprising. However, speaking practically the author does not know an example of a Mellin transform $L(s,N)$ which has a polynomial bound on average but not pointwise. For this reason, it suffices to cite \cite[Lemma B.4]{pierce-turnage-butterbaugh-zaman2025guide} in all applications we are aware of. Nevertheless, we have opted to include a proof of Lemma \ref{lem:contour_shift} because it fits thematically with our focus on leveraging the twisted average.

    We also note that \cite[Lemma B.4]{pierce-turnage-butterbaugh-zaman2025guide} is stated for an arbitrary smoothing function, and Lemma \ref{lem:contour_shift} can certainly be proven at this level of generality using a similar argument.
\end{remark}

\begin{remark}
    In case $N(X) = \sum_{\lambda_n \le X} a_n$ is the summatory function of the coefficients of a Dirichlet series, the smoothed function $A_N^k(X)$ agrees with the Riesz mean
    \[
        \frac{1}{\Gamma(k+1)}\sum_{\lambda_n \le x}a_n(X-\lambda_n)^k
    \]
    appearing in \cite{chandrasekharan1962functional,lowry-duda-taniguchi-thorne2022uniform}.
\end{remark}

\begin{proof}[Proof of Lemma \ref{lem:contour_shift}]
This result is proved, as usual, by shifting the contour in Perron's formula
\[
    A_N^k(X) = \frac{1}{2\pi i} \int_{c-i\infty}^{c+i\infty} L(s,N) \frac{\Gamma(s)}{\Gamma(s+k+1)}X^{s+k}ds,
\]
which holds for each integer $k > 0$ and each $c > \sigma_a$.


In order to avoid the need for pointwise bounds for $|L(s,N)|$ in this argument, we include an auxiliary average in Perron's formula
\[
    A_N^k(X) = \lim_{T\to \infty}\frac{1}{2\pi i T} \int_T^{2T} \int_{c-i\infty}^{c+i\infty} L(s,N) \frac{\Gamma(s)}{\Gamma(s+k+1)}X^{s+k} ds du.
\]
As stated, this is trivial because the variable of integration $u$ does not actually appear in the integrand. However, when we shift the contour to the left we shift in a way that depends on $u$:
\begin{align*}
    \int_{c-i\infty}^{c+i\infty} = &\int_C + \underbrace{\int_{c-i\infty}^{c-iu} + \int_{c+iu}^{c+i\infty}}_{\text{outer vertical pieces}} + \underbrace{\int_{c-iu}^{\sigma_p - \delta -iu} + \int_{\sigma_p - \delta +iu}^{c+i u}}_{\text{horizontal pieces}}\\
    &+ \underbrace{\int_{\sigma_p-\delta-iu}^{\sigma_p - \delta -iT_0} + \int_{\sigma_p - \delta +iT_0}^{\sigma_p-\delta+i u}}_{\text{inner vertical pieces}} + \underbrace{\int_{\gamma}}_{\substack{\text{small detour}\\\text{around poles}}}.
\end{align*}
The integral over $C$ is independent of the variable $u$, seeing only the residues, while each of the remaining integrals are part of the shifted contour which is depicted in \Cref{fig:shifted_contour}. The auxiliary average now behaves as an average over several different ways of shifting the contour integral.

\begin{figure}[ht]
\begin{center}
    \begin{tikzpicture}

    \draw[gray!50, thin, step=0.5] (-1,-4) (4,4);
    \draw[very thick,->] (-1,0) -- (4.2,0) ;
    \draw[very thick,->] (0,-4) -- (0,4.2) ;

    \draw (3.5,0.05) -- (3.5,-0.05) node[below] {\tiny$c$};
    \draw (3,0.05) -- (3,-0.05) node[below] {\tiny$\sigma_a$};
    \draw[dashed] (3,-4) -- (3,4);
    
    \draw (1.5,0.05) -- (1.5,-0.05) node[below right] {\tiny$\sigma_p - \delta$};
    \draw (-0.05,3) -- (0.05,3) node[left] {\tiny$u$};
    \draw (-0.05,-3) -- (0.05,-3) node[left] {\tiny$-u$};
    \draw (-0.05,0.5) -- (0.05,0.5) node[left] {\tiny$T_0$};
    \draw (-0.05,-0.5) -- (0.05,-0.5) node[left] {\tiny$-T_0$};

    \draw[<-] (3.5,4) -- (3.5,3) -- (1.5,3) -- (1.5,0.5);
    \draw[<-] (3.5,-4) -- (3.5,-3) -- (1.5,-3) -- (1.5,-0.5);
    
    \draw[-] (1.5,-0.5) to[bend left] (1.5,0.5);
    \draw[dashed] (1.5,-0.5) -- (1.5,0.5);
    \draw (1.4,0) node[above left] {\tiny$\gamma$};
    
    \end{tikzpicture}
\end{center}
\caption{\label{fig:shifted_contour}}
\end{figure}

The integrals over $C$ and $\gamma$ are independent of $u$ so that
\[
    \lim_{T\to \infty}\frac{1}{T}\int_T^{2T} \left(\int_C+\int_\gamma\right) = \int_C + \int_\gamma,
\]
as they appear in the statement. It now suffices to evaluate the limit for the remaining contour integrals.

\subsubsection{The outer vertical integrals}
In this section we evaluate
\[
    \lim_{T\to \infty}\frac{1}{2\pi i T} \int_T^{2T}\int_{c\pm iu}^{c\pm i\infty} L(s,N) \frac{\Gamma(s)}{\Gamma(s+k+1)} X^{s+k} ds du.
\]
Restricting to $k \ge 2$ we bound
\begin{align*}
    \left\lvert\frac{1}{2\pi i T} \int_T^{2T}\int_{c\pm iu}^{c\pm i\infty} L(s,N) \frac{\Gamma(s)}{\Gamma(s+k+1)} X^{s+k} ds du\right\rvert &\ll_{N,X,c,k} \frac{1}{T}\int_T^{2T}\left\lvert\int_{\pm u}^{\pm \infty} \frac{1}{t^{k+1}} dt\right\rvert du\\
    &\ll_{N,X,c,k} T^{1-k},
\end{align*}
which tends to $0$ as $T\to \infty$. Thus, the outer vertical integrals vanish in the limit.

\subsubsection{The horizontal integrals}
In this section we evaluate
\[
    \lim_{T\to \infty}\frac{1}{2\pi i T} \int_T^{2T}\int_{\sigma_p-\delta \pm iu}^{c\pm iu} L(s,N) \frac{\Gamma(s)}{\Gamma(s+k+1)}X^{s+k} ds du.
\]
The main idea here is that the contour integral is along a real line while the auxiliary average is essentially an integral along the imaginary line, so that the order of integration can be painlessly swapped.
\begin{align*}
    &\left\lvert\frac{1}{2\pi i T}\int_T^{2T}\int_{\sigma_p-\delta \pm iu}^{c\pm iu} L(s,N) \frac{\Gamma(s)}{\Gamma(s+k+1)}X^{s+k} ds du\right\rvert\\
    &\le \frac{1}{2\pi T} \int_{\sigma_p-\delta}^c X^{\sigma+k}\left\lvert\int_T^{2T} L(\sigma\pm iu,N)\frac{\Gamma(\sigma\pm iu)}{\Gamma(\sigma\pm iu+k+1)}X^{\pm iu} du \right\rvert d\sigma.
\end{align*}

The inner integral is bounded using integration by parts. For simplicity's sake let
\[
    f_{k,\sigma}(u) := \frac{\Gamma(\sigma\pm iu)}{\Gamma(\sigma\pm iu+k+1)} = \frac{1}{(\sigma\pm iu)(\sigma \pm iu+1)\cdots(\sigma\pm iu + k)}
\]
so that
\[
    f_{k,\sigma}'(u) = \frac{\mp i}{(\sigma\pm iu)(\sigma \pm iu+1)\cdots(\sigma\pm iu + k)}\sum_{j=0}^k \frac{1}{\sigma\pm iu +j}.
\]
We can directly bound each of these by
\begin{align*}
    |f_{k,\sigma}(u)| &\le u^{-k-1},\\
    |f_{k,\sigma}'(u)| &\le ku^{-k-2}.
\end{align*}
Taking these bounds together with the bound
\[
    \int_{T_0}^T L(\sigma \pm it,N)X^{\pm it} dt \ll T^{\widetilde{\eta}}
\]
we can apply integration by parts to show that
\begin{align*}
    &\left\lvert\frac{1}{T}\int_T^{2T} L(\sigma\pm iu,N)X^{\pm iu} f_{k,\sigma}(u) du\right\rvert\\
    &\ll \frac{1}{T}f_{k,\sigma}(u)u^{\widetilde{\eta}} \Big{|}_T^{2T} + \frac{1}{T}\int_T^{2T} u^{\widetilde{\eta}} f_{k,\sigma}'(u) du\\
    &\ll T^{\widetilde{\eta}-k-2}.
\end{align*}
Taking the outer integral over $\sigma$, we have now proven
\begin{align*}
    \left\lvert\frac{1}{2\pi i T}\int_T^{2T}\int_{\sigma_p-\delta \pm iu}^{c\pm iu} L(s,N) \frac{\Gamma(s)}{\Gamma(s+k+1)} X^{s+k}ds du\right\rvert \ll T^{\widetilde{\eta} - k - 2},
\end{align*}
allowing the implied constant to depend on $X$. Restricting to $k \ge \widetilde{\eta} - 2$, this term also vanishes as $T$ tends towards infinity.

\subsubsection{The inner vertical integrals}

Using integration by parts as in the previous section, we know that integral $\int_{\sigma_p-\delta\pm iT_0}^{\sigma_p - \delta \pm i\infty}$ converges absolutely when $k \ge \widetilde{\eta} - 2$, similarly to the horizontal integrals.

In order to evaluate the limit, we switch the order of integration
\begin{align*}
    \frac{1}{T}\int_T^{2T} \int_{\sigma_p - \delta +iT_0}^{\sigma_p - \delta + iu} &=  \frac{1}{T}\int_{\sigma_p - \delta +iT_0}^{\sigma_p - \delta + i2T} \int_{\max\{t,T\}}^{2T}
\end{align*}
Recalling that the integrand is independent of the variable of integration $u$ for the auxiliary average, we can directly evaluate the inner integral to get
\begin{align*}
    &= \int_{\sigma_p - \delta +iT_0}^{\sigma_p - \delta + i2T}\left(2 - \max\{t/T,1\}\right)\\
    &= \int_{\sigma_p - \delta +iT_0}^{\sigma_p - \delta + iT} + \left(\int_{\sigma_p - \delta +iT}^{\sigma_p - \delta + i2T}\left(2 - \max\{t/T,1\}\right)\right)\\
    &= \int_{\sigma_p - \delta +iT_0}^{\sigma_p - \delta + iT} + O\left(\int_{\sigma_p - \delta +iT}^{\sigma_p - \delta + i2T}\right).
\end{align*}
The limit of the first term converges to $\int_{\sigma_p-\delta+ iT_0}^{\sigma_p - \delta + i\infty}$, while the integral in the big-oh is contained in the tail of this improper integral, so it tends to $0$ in the limit. Thus we have shown
\[
    \lim_{T\to \infty} \frac{1}{T}\int_T^{2T} \int_{\sigma_p - \delta +iT_0}^{\sigma_p - \delta + iu} = \int_{\sigma_p - \delta +iT_0}^{\sigma_p - \delta + i\infty}.
\]
The analogous result holds for the negative imaginary part integral by the same argument. Thus,
\[
    \lim_{T\to \infty} \frac{1}{T}\int_T^{2T} \left(\int_{\sigma_p - \delta +iT_0}^{\sigma_p - \delta + iu} + \int_{\sigma_p - \delta -iu}^{\sigma_p - \delta - iT_0} + \int_{\gamma}\right) = \int_{(\sigma_p-\delta)}^*.
\]
Combined with the above limits, this concludes the proof.
\end{proof}

\section{Proof of the Main Theorem via Unsmoothing}\label{sec:unsmoothing}

For any $y>0$, define the difference and integral operators
\begin{align*}
    \Delta_y f&: x\mapsto f(x+y) - f(x) & \Delta_y^{(k)} = \underbrace{\Delta_y\circ \cdots \circ\Delta_y}_k\\
    \widetilde{\Delta}_y f&: x\mapsto \int_x^{x+y}f(u)du & \widetilde{\Delta}_y^{(k)} = \underbrace{\widetilde{\Delta}_y\circ \cdots \circ\widetilde{\Delta}_y}_k.
\end{align*}
The key insight to Landau's method of finite differencing is that
\[
    y^{-k}\Delta_y^{(k)} A_N^k(X) \approx A_N^0(X) = N(X),
\]
where the error in this approximation can be controlled. We will largely follow Landau's method in a similar vein to \cite{roux2011,alberts2024power}, where we use integration by parts to incorporate the average bounds in place of pointwise bounds. We take additional care to track the dependence on various parameters during the course of the argument, similar to \cite{lowry-duda-taniguchi-thorne2022uniform}.

\subsection{Unsmoothing}

We separate the integrals appearing in Lemma \ref{lem:contour_shift} into outer vertical pieces and inner vertical pieces:
\begin{align*}
    \int_{C} + \int_{(\sigma_p - \delta)}^* = &\int_C + \underbrace{\int_{\sigma_p-\delta-i\infty}^{\sigma_p-\delta-iT} + \int_{\sigma_p-\delta+iT}^{\sigma_p-\delta+i\infty}}_{\text{outer vertical pieces}} + \underbrace{\int_{\sigma_p-\delta-iT}^{\sigma_p - \delta -iT_0} + \int_{\sigma_p - \delta +iT_0}^{\sigma_p-\delta+i T}}_{\text{inner vertical pieces}} + \underbrace{\int_{\gamma}}_{\substack{\text{small detour}\\\text{around poles}}}.
\end{align*}
In the following subsections, we will bound the contribution of
\[
    y^{-k}\Delta_y^{(k)} \left[\int_j \cdots\right]
\]
for each piece $\int_j$ of the contour appearing above, which all together will give an asymptotic expansion for $y^{-k}\Delta_y^{(k)} A_N^k(X)$ and correspondingly for $N(X)$. The values for $y$ and $k$ will be chosen optimally at the end of this section to prove Theorem \ref{thm:main}, giving a bound in terms of $X$, $T$, and the various parameters associated to $N$.

\subsubsection{The main terms}
In this section we evaluate the expression
\[
    y^{-k}\Delta_y^{(k)}\left[\frac{1}{2\pi i}\int_C L(s,N) \frac{\Gamma(s)}{\Gamma(s+k+1)}X^{s+k} ds\right],
\]
which will correspond to the residue terms in Theorem \ref{thm:main}, i.e.~the main asymptotic terms. This term is exactly the same as the corresponding integral appearing in \cite{chandrasekharan1962functional,roux2011,lowry-duda-taniguchi-thorne2022uniform,alberts2024power}, where we adopt notation similar to that in \cite{chandrasekharan1962functional,lowry-duda-taniguchi-thorne2022uniform},
\[
    S_{N,\alpha}^k(X) := \frac{1}{2\pi i}\int_{C_{\alpha}} L(s,N) \frac{\Gamma(s)}{\Gamma(s+k+1)}X^{s+k} ds,
\]
and it will suffice to analyze it in the same way. We include the full argument because the proof is short and we can be careful to keep track of the relevant dependencies.

The difference operator commutes with the integrals, so that
\[
    y^{-k}\Delta_y^{(k)}\left[S_{N,\sigma_p-\delta}^k(X)\right] = y^{-k}\frac{1}{2\pi i}\int_{C} L(s,N) \frac{\Gamma(s)}{\Gamma(s+k+1)}\Delta_y^{(k)}\left[X^{s+k}\right] ds.
\]
The Fundamental Theorem of Calculus implies that $\Delta_y^{(k)} f(X) = \widetilde{\Delta}_y^{(k)} f^{(k)}(X)$. Differentiating $k$ times and pulling the integral operator back to the outside gives
\begin{align*}
    = y^{-k}\widetilde{\Delta}_y^{(k)}\left[\frac{1}{2\pi i} \int_{C} L(s,N) \frac{X^{s}}{s} ds\right] = y^{-k}\widetilde{\Delta}_y^{(k)}\left[S_{N,\sigma_p-\delta}^0(X)\right].
\end{align*}
For any function $f(x)$, the mean value theorem implies
\begin{align*}
    \left\lvert y^{-k}\widetilde{\Delta}_y^{(k)}[f(u)] - f(X)\right\rvert &= \left\lvert y^{-k}\widetilde{\Delta}_y^{(k)}[f(u) - f(X)]\right\rvert\\
    &\le ky\cdot \max_{x\in [X,X+ky]} |f'(x)|.
\end{align*}
In our case, $S_{N,\alpha}^0(X) = \sum X^zR_z(\log X)$. The derivative of any single term is
\[
    \frac{d}{dX}[X^z(\log X)^n] = X^{z-1}(\log X)^n\left(z + \frac{n}{\log X}\right)\ll_{z,n} X^{z-1}(\log X)^n.
\]
Thus, we can bound $|(S_{N,\alpha}^0)'(x)|\ll x^{-1}R_{N,\alpha}(x)$.

This implies the $\int_C$ term satisfies
\[
    \left\lvert y^{-k}\Delta_y^{(k)}\left[S_{N,\sigma_p-\delta}^k(X)\right] - S_{N,\sigma_p-\delta}^0(X)\right\rvert \ll ky\left(\max_{x\in [X,X+ky]} x^{-1}R_{N,\sigma_p-\delta}(x)\right),
\]
where the implied constant depends only on the locations and orders of the poles of $L(s,N)/s$.


\subsubsection{The inner vertical integrals}
In this section we bound
\[
    y^{-k}\Delta_y^{(k)}\left[\frac{1}{2\pi i} \int_{\sigma_p-\delta \pm iT_0}^{\sigma_p-\delta\pm iT} L(s,N) \frac{\Gamma(s)}{\Gamma(s+k+1)}X^{s+k} ds\right].
\]
Once again, the difference operator commutes with the integrals so that the Fundamental Theorem of Calculus implies the above term is equal to
\[
    = y^{-k}\widetilde{\Delta}_y^{(k)}\left[\frac{1}{2\pi i} \int_{\sigma_p-\delta \pm iT_0}^{\sigma_p-\delta\pm iT} L(s,N) \frac{X^{s}}{s} ds\right].
\]
Integration by parts implies
\begin{align*}
    &\int_{\sigma_p-\delta \pm iT_0}^{\sigma_p-\delta\pm iT} L(s,N) X^{it}\frac{1}{s} ds\\
    &= \frac{1}{\sigma_p-\delta \pm it}\left(\int_{T_0}^t L(\sigma_p-\delta \pm iy,N)X^{\pm iy} dy\right)\Big{|}_{T_0}^{T}\\
    &\hspace{0.5cm}+ \int_{T_0}^{T} \left(\int_{T_0}^t L(\sigma_p-\delta\pm iy,N)X^{\pm iy} dy\right) \frac{1}{(\sigma_p-\delta\pm i t)^2} dt\\
    &\le QT^{\eta-1}(\log T)^{\beta} + Q\int_{T_0}^T t^{\eta - 2}(\log t)^\beta dt.
\end{align*}
We give separate bounds for the integral, depending on the value of $\eta$. We bound the integral above by $\int_{e}^T$ to show that
\begin{align*}
    &\le QT^{\eta - 1}((\log T)^{\beta} + g_{\eta,\beta}(T)) + Q h_{\eta,\beta},
\end{align*}
where $g_{\eta,\beta}(u)$ and $h_{\eta,\beta}(u)$ are given by
\begin{align*}
    g_{\eta,\beta}(u) &= \begin{cases}
        (\log u)^{\beta+1} & \eta = 1\\
        \frac{1}{|\eta-1|}(\log u)^{\beta} & \eta \ne 1,
    \end{cases} &
    h_{\eta,\beta} &= \begin{cases}
        1 & \eta = 1\\
        \frac{1}{|\eta-1|}e^{\eta-1} & \eta \ne 1.
    \end{cases} &
\end{align*}

Plugging this in gives the bound
\begin{align*}
    &\left\lvert\frac{1}{2\pi i}\int_{\sigma_p-\delta \pm iT_0}^{\sigma_p-\delta\pm iT} L(s,N) \frac{X^{s}}{s} ds du\right\rvert \le X^{\sigma_p-\delta}\left(QT^{\eta - 1}((\log T)^{\beta} + g_{\eta,\beta}(T)) + Q h_{\eta,\beta}\right).
\end{align*}
We can bound the integrals $\widetilde{\Delta}_y^{(k)} f(X) \le y^k \sup\{|f(x)|: X\le x \le X+ky\}$. This implies the contribution of the inner vertical pieces is
\[
    \le \left(\max_{x\in \{X,X+ky\}} x^{\sigma_p - \delta}\right)\left(QT^{\eta - 1}((\log T)^{\beta} + g_{\eta,\beta}(T)) + Q h_{\eta,\beta}\right),
\]
as $x^{\sigma_a-\delta}$ is monotonic so that the maximum value is necessarily achieved at one of the endpoints.

\subsubsection{The outer vertical integrals}
In this section we bound
\[
    y^{-k}\Delta_y^{(k)}\left[\frac{1}{2\pi i} \int_{\sigma_p-\delta \pm iT}^{\sigma_p-\delta\pm i\infty} L(s,N) \frac{\Gamma(s)}{\Gamma(s+k+1)}X^{s+k} ds\right].
\]
This time we do not move the difference operator, and instead bound the integral directly analogously to how we bounded the horizontal integrals during the proof of \Cref{lem:contour_shift}. Recall the notation $f_{k,\sigma}$ from that argument.

Applying integration by parts to the inner integral, we find that
\begin{align*}
    &\left\lvert\int_{T}^{\infty} L(\sigma_p -\delta \pm it,N) X^{\pm it}f_{k,\sigma_a-\delta}(u) ds\right\rvert\\
    &\le f_{k,\sigma_p-\delta}(t) Q t^{\eta}(\log t)^{\beta}\Big{|}_{T}^{\infty} + \int_T^{\infty} Qy^{\eta}(\log y)^{\beta}f_{k,\sigma_p-\delta}'(y)dy\\
    &\le QT^{\eta-k-1}(\log T)^{\beta} + \frac{k}{k+1-\eta}QT^{\eta-k-1}(\log T)^{\beta}.
\end{align*}
Restricting to $k \ge 3\eta - 3$, the fraction is bounded above by $3/2$.

Plugging this in to the outer integral and noting that $\frac{5}{4\pi} < 1$ we produce the upper bound
\[
    \le QX^{\sigma_p-\delta+k}T^{\eta-k-1}(\log T)^{\beta}.
\]
Bounding the difference operator shows that this integral piece is bounded above by
\[
    \le 2^k y^{-k} Q\left(\max_{x\in \{X,X+ky\}} x^{\sigma_p-\delta+k}\right)T^{\eta-k-1}(\log T)^{\beta}.
\]

\subsubsection{The small contour avoiding poles}
In this section we bound
\[
    y^{-k}\Delta_y^{(k)}\left[\frac{1}{2\pi i} \int_{\gamma} L(s,N) \frac{\Gamma(s)}{\Gamma(s+k+1)}X^{s+k} ds du\right].
\]
Similar to previous sections, we can move the difference operator inside the integrals, apply the Fundamental Theorem of Calculus, then pull the integral operator to the outside to get
\[
    = y^{-k}\widetilde{\Delta}_y^{(k)}\left[\frac{1}{2\pi i}\int_{\gamma} L(s,N) \frac{X^{s}}{s} ds \right].
\]
By the triangle inequality, it follows that
\begin{align*}
    \le y^{-k}\widetilde{\Delta}_y^{(k)}\left[\frac{1}{2\pi}\int_\gamma |L(s,N)| \frac{X^{{\rm Re}(s)}}{|s|} ds \right]
\end{align*}
We chose $\gamma$ to be entirely to the left of the line ${\rm Re}(s) = \sigma_p - \delta$, so this term is further bounded by
\[
    \le \int_\gamma\left\lvert\frac{L(s,N)}{s}\right\rvert ds\left(\max_{x\in \{X,X+ky\}}x^{\sigma_p-\delta}\right).
\]

\subsubsection{Conclusion of the contour shifting}\label{subsubsec:contour_shift_result}
Taking all the pieces of the shifted contour together, we have proven that
\[
    \left\lvert y^{-k}\Delta_y^{(k)}\left[A_N^k(X)\right] - S_{N,\sigma_p-\delta}^0(X)\right\rvert
\]
is bounded above in absolute value by
\begin{align*}
    &\ll ky\left(\max_{x\in [X,X+ky]} x^{-1}R_{N,\sigma_p-\delta}(x)\right) + 2^k y^{-k}(X+ky)^{\sigma_p-\delta+k}T^{\eta-k-1}(\log T)^{\beta}\\
    &\hspace{0.5cm}+ \left(\max_{x\in \{X,X+ky\}} x^{\sigma_p-\delta}\right)Q\left(T^{\eta - 1}((\log T)^{\beta} + g_{\eta,\beta}(T)) + h_{\eta,\beta}\right)\\
    &\hspace{0.5cm}+ \int_\gamma\left\lvert\frac{L(s,N)}{s}\right\rvert ds\left(\max_{x\in \{X,X+ky\}} x^{\sigma_p-\delta}\right),
\end{align*}
where we have picked up the constraints $k\ge 2$ and $k \ge 3\eta - 3$ and the implied constant depends only on the locations and orders of the poles of $L(s,N)/s$.

\subsection{The error bound before optimization}

Suppose now that $N(X)$ is nondecreasing. One has the bounds
\[
    A_N^0(X) \le y^{-k}\Delta_y^{(k)}\left[A_N^k(X)\right] \le A_N^0(X+ky)
\]
as in \cite[Equation 4.15]{chandrasekharan1962functional}. On the one hand, this implies
\[
    A_N^0(X) - S_{N,\sigma_p-\delta}^0(X) \le y^{-k}\Delta_y^{(k)}\left[A_N^k(X)\right] - S_{N,\sigma_p-\delta}^0(X).
\]
On the other hand, if we restrict to $ky < X$ then
\begin{align*}
    &A_N^0(X) - S_{N,\sigma_p-\delta}^0(X)\\
    &\ge y^{-k}\Delta_y^{(k)}\left[A_N^k(X-ky)\right] - S_{N,\sigma_p-\delta}^0(X)\\
    &= y^{-k}\Delta_y^{(k)}\left[A_N^k(X-ky)\right] - S_{N,\sigma_p-\delta}^0(X-ky) + O\left(ky\max_{x\in [X-ky,X]} x^{-1}R_{N,\sigma_p-\delta}(x)\right).
\end{align*}
The restriction to $ky<X$ is necessary for $S_{N,\sigma_p-\delta}^0(X-ky)$ to be well-defined, as $(X-ky)^s$ for complex numbers $s$ is only well-defined if $X-ky > 0$.

Suppose that the bounds for $\int_{T_0}^T L(\sigma\pm it,N)Z^{\pm it}dt$ hold for $Z\in \{X,X-ky\}$. Then by a change of variables between $X$ and $X-ky$, we conclude that
\[
    A_N^0(X) - S_{N,\sigma_p-\delta}^0(X)
\]
is bounded in absolute value by
\begin{align*}
    &\ll ky\left(\max_{x\in [X-ky,X+ky]} x^{-1}R_{N,\sigma_p-\delta}(x)\right) + 2^k y^{-k}(X+ky)^{\sigma_p-\delta+k}T^{\eta-k-1}(\log T)^{\beta}\\
    &\hspace{0.5cm}+ \left(\max_{x\in \{X-ky,X+ky\}} x^{\sigma_p-\delta}\right)Q\left(T^{\eta - 1}((\log T)^{\beta} + g_{\eta,\beta}(T)) + h_{\eta,\beta}\right)\\
    &\hspace{0.5cm}+ \int_\gamma\left\lvert\frac{L(s,N)}{s}\right\rvert ds\left(\max_{x\in \{X-ky,X+ky\}} x^{\sigma_p-\delta}\right).
\end{align*}
We slightly strengthen our constraint to $ky \le \frac{1}{2}X$, so that the maximums can be more easily bounded. All together, we have proven the following unoptimized error bound:

\begin{theorem}\label{thm:unoptimized}
    Let $N$ be a nondecreasing function which satisfies the assumptions of \Cref{subsec:hypotheses}. Suppose the bounds for $\int_{T_0}^T L(\sigma \pm it,N)Z^{it} dt$ given above \Cref{thm:main} hold for $Z\in \{X,X-ky\}$. For each integer $k\ge \max\{2,\widetilde{\eta}-2,3\eta-3\}$, each real number $T \ge T_0$, and each real number $y > 0$ for which $ky\le \frac{1}{2}X$, the difference
    \[
        A_N^0(X) - S_{N,\sigma_p-\delta}^0(X)
    \]
    is bounded in absolute value by
    \begin{align*}
        &\ll ky X^{-1}R_{N,\sigma_p-\delta}(X) + 3^k y^{-k}X^{\sigma_a-\delta+k}T^{\eta-k-1}(\log T)^{\beta} + 2^{\delta}QX^{\sigma_p-\delta}T^{\eta-1}\left((\log T)^{\beta} + g_{\eta,\beta}(T)\right)\\
        &\hspace{0.5cm} + 2^{\delta}\left(Qh_{\eta,\beta} + \int_\gamma\left\lvert\frac{L(s,N)}{s}\right\rvert ds\right) X^{\sigma_p-\delta},
    \end{align*}
    where the implied constant depends only on $\sigma_p$, and the locations and orders of the poles of $L(s,N)/s$.
\end{theorem}

The dependence on $\sigma_p$ is not too complicated, coming from the factors like $(\frac{3}{2}X)^{\sigma_p-\delta}$, which we have hidden in the implied constant only to simplify the expression as much as possible. This is not so significant a burden, as $N$ nondecreasing implies that the abscissa of absolute convergence of $L(s,N)$ is also a pole of $L(s,N)$. For this reason, $\sigma_p$ can often be taken to be part of the data of ``locations of the poles''.

If $N$ is not nondecreasing, but $|N(X)-N(Y)| \le \widehat{N}(X)-\widehat{N}(Y)$ for each $X\ge Y\ge 0$ and some nondecreasing function $\widehat{N}$, a similar argument proves the analogous result to \Cref{thm:unoptimized} with a sum of the contribution from both $N$ and $\widehat{N}$. The key step is \cite[Proposition 16.16]{roux2011} (also stated as \cite[Proposition 6.5]{alberts2024power}), which implies that
\[
    y^{-k}\Delta_y^{(k)} A_N^k(X) - A_N^0(X) = O\left(y^{-k}\Delta_y^{(k)}\left[A_{\widehat{N}}^k(X)\right] - A_{\widehat{N}}^0(X)\right).
\]
Tracing through the proof shows that the implied constant can be taken equal to $1$.

Thus, we can bound
\begin{align*}
    |A_N^0(X) - S_{N,\sigma_p-\delta}^0(X)| &\le |y^{-k}\Delta_y^{(k)}\left[A_N^k(X)\right] - S_{N,\sigma_p-\delta}^0(X)|+|y^{-k}\Delta_y^{(k)}\left[A_N^k(X)\right] - A_N^0(X)|\\
    &\le |y^{-k}\Delta_y^{(k)}\left[A_N^k(X)\right] - S_{N,\sigma_p-\delta}^0(X)|+|y^{-k}\Delta_y^{(k)}\left[A_{\widehat{N}}^k(X)\right] - A_{\widehat{N}}^0(X)|
\end{align*}
where the first term is bounded by the results of \Cref{subsubsec:contour_shift_result}, while the second term is bounded by applying \Cref{thm:unoptimized} for $\widehat{N}$ in conjunction with the results of \Cref{subsubsec:contour_shift_result}.

\subsection{Optimizing the error}

We optimize the error bound in the case that $N$ is nondecreasing. The case that $|N(X)-N(Y)| \le \widehat{N}(X)-\widehat{N}(Y)$ for each $X\ge Y\ge 0$ and some nondecreasing $\widehat{N}$ is completely analogous, where we use a different parameter $\widehat{T}$ to measure to perform the contour integration producing the bound for $y^{-k}\Delta_y^{(k)}\left[A_{\widehat{N}}^k(X)\right] - A_{\widehat{N}}^0(X)$. This part of the error is bounded by the error contribution for $\widehat{N}$ following from \Cref{thm:unoptimized}, which is optimized with the same choices as below.

Now we choose values for $y$ and $k$ in \Cref{thm:unoptimized} in order to conclude Theorem \ref{thm:main}. One can write down the optimal choices for one parameter at a time, but these get increasingly complicated and in particular make the constraints difficult to check. We opt to make simpler choices that are optimal up to order of magnitude to avoid these issues.

We first choose
\[
    y = \frac{3}{k}XT^{-1}.
\]
This satisfies the constraint $ky\le \frac{1}{2}X$ if and only if $T \ge 6$. Moreover, $X-ky=X(1-3T^{-1})$ agrees with the hypotheses for bounding $\int_{T_0}^T L(\sigma\pm it,N)Z^{\pm it}dt$ above \Cref{thm:main}. Plugging this in to the error bounds in \Cref{thm:unoptimized}, we give an error bound of the form
\begin{align*}
    &\ll T^{-1}R_{N,\sigma_p-\delta}(X) + k^kX^{\sigma_p-\delta}T^{\eta-1}(\log T)^{\beta} + 2^{\delta}QX^{\sigma_p-\delta}T^{\eta-1}\left((\log T)^{\beta} + g_{\eta,\beta}(T)\right)\\
    &\hspace{0.5cm} + 2^{\delta}\left(Qh_{\eta,\beta}+\int_\gamma\left\lvert\frac{L(s,N)}{s}\right\rvert ds\right) X^{\sigma_p-\delta}.
\end{align*}
This has almost completely removed the dependence on $k$, save for the constant $k^k$. It then suffices to take $k$ as small as possible, which is $k = \lceil \max\{2,\widetilde{\eta}-2,3\eta-3\}\rceil$.

\Cref{thm:main} then follows after factoring out $X^{\sigma_p-\delta}$ from the error terms and evaluating $g_{\eta,\beta}(T)$ and $h_{\eta,\beta}$.

\section{Counting Abelian Number Fields}

We are now ready to prove Corollary \ref{cor:uncond_square_root_savings}. For any abelian group $G$, write $G_d$ for the set of elements of order exactly $d$. The generating Dirichlet series of $G$-\'etale discriminants over $\Q$ is given by
\[
    D_{\Q}^{\text{\'et}}(G;s) = F(s)\prod_{p\nmid |G|} \left(\sum_{d} |G_d|\mathbf{1}_{1,d}(p) p^{-|G|(1-1/d)s}\right),
\]
where $\mathbf{1}_{a,d}$ is the characteristic function of the set of integers which are congruent to $a \pmod d$ and $F(s)$ is a Dirichelt polynomial given by the finite product of contributions from wild primes (and so is bounded on the right half-plane ${\rm Re}(s) > 0$). This follows directly from work of Wright \cite{wright1989}, as well as \cite[Section 3]{alberts2024power}, due to $\Q$ having trivial class group and $\mathbf{1}_{1,d}(p)$ picking out exactly which primes can be ramified in an abelian extension of $\Q$ with tame inertia group of order $d$. The sum over $d$ is supported on $d$ dividing the exponent of $G$, so for the remainder of this section we keep this restriction.

\subsection{Meromorphic Continuation}

We require an explicit meromorphic continuation of $D_{\Q}^{\text{\'et}}(G;s)$. This is done by \cite[Corollary 3.3]{alberts2024power} up to the half-plane ${\rm Re}(s) > 1/2a(G)$. However, in order to get square root savings we require a meromorphic continuation to ${\rm Re}(s) \ge 1/2a(G)$.

\begin{proposition}\label{prop:abelian_meromorphic_continuation}
    Let $G$ be an abelian group. If $\ell$ is the smallest prime dividing $|G|$, then $D_\Q^{\text{\'et}}(G;s)$ is equal to
    \begin{align*}
        \left(\prod_{d > 1} \zeta_{\Q(\zeta_{d})}\left(|G|(1-1/d)s\right)^{\frac{|G_d|}{\phi(d)}}\right)\zeta_{\Q(\zeta_\ell+\zeta_\ell^{-1})}(2a(G)s)^{-\frac{|G_\ell|}{2}}\zeta_{\Q(\zeta_\ell)}(2a(G)s)^{-\frac{2\ell-3}{2\ell-2}|G_\ell|^2 + \frac{\ell-2}{2}|G_\ell|}H(s),
    \end{align*}
    where $H(s)$ is an Euler product which converges absolutely on the right half-plane ${\rm Re}(s) > \frac{1}{2a(G)+1}$.
\end{proposition}

We note that if $\ell=2$ then $\zeta_{\Q(\zeta_\ell)} = \zeta_{\Q(\zeta_\ell+\zeta_\ell^{-1})} = \zeta$, so the corresponding terms may be combined. The process of meromorphically continuing Euler products by factoring out powers of $L$-functions is outlined in the expository paper \cite{alberts2024explicit}. We feel it is clearest to give a direct proof of \Cref{prop:abelian_meromorphic_continuation} in this case, but one can alternatively apply \cite[Corollary 14.5]{alberts2024explicit} to reproduce the same result using a recursive formula for the powers of $L$-functions that appear.

\begin{proof}
    Throughout this argument, we let $H(s)$ denote an Euler product which converges absolutely on the right half-plane ${\rm Re}(s) > \frac{1}{2a(G)+1}$. In particular, $H(s)$ need not denote the same such function in every instance, and is allowed to absorb constants and other absolutely convergent Euler products on this region.
    
    The result can be checked directly by computing the Euler factors of
    \[
        D_{\Q}^{\text{\'et}}(G;s)\left(\prod_{d > 1} \zeta_{\Q(\zeta_{d})}\left(|G|(1-1/d)s\right)^{-\frac{|G_d|}{\phi(d)}}\right)\zeta_{\Q(\zeta_\ell+\zeta_\ell^{-1})}(2a(G) s)^{\frac{|G_\ell|}{2}}\zeta_{\Q(\zeta_\ell)}(2a(G)s)^{|G_\ell|^2\frac{2\ell-3}{2\ell-2} - |G_\ell|\frac{\ell-2}{2}}.
    \]
    Indeed, for $d > \ell$ dividing $|G|$, we note that $|G|(1-1/d)$ is an integer which satisfies
    \begin{align*}
        |G|(1-1/d) &< 2|G|(1-1/\ell) = 2a(G) &&\Rightarrow |G|(1-1/d) < 2a(G) + 1\\
        2|G|(1-1/d) &> 2|G|(1-1/\ell) = 2a(G) &&\Rightarrow 2|G|(1-1/d) \ge 2a(G)+1
    \end{align*}
    Multiplying out the Euler product for the Dedekind zeta functions $\zeta_{\Q(\zeta_d)}$ gives
    \begin{align*}
        \zeta_{\Q(\zeta_d)}(|G|(1-1/d)s)^{-\frac{|G_d|}{\phi(d)}} = H(s)\prod_{p\nmid |G|} \left(1 - |G_d|\mathbf{1}_{1,d}(p)p^{-|G|(1-1/d)s} + O\left(p^{-(2a(G)+1)s}\right)\right),
    \end{align*}
    Taking advantage of the fact that $d > \ell$ implies $|G|(1-1/d) + a(G) \ge 2a(G) + 1$, this implies
    \[
        D_{\Q}^{\text{\'et}}(G;s)\left(\prod_{d > \ell} \zeta_{\Q(\zeta_{d})}\left(|G|(1-1/d)s\right)^{-\frac{|G_d|}{\phi(d)}}\right) = H(s)\prod_{p\nmid |G|}\left(1 + |G_\ell|\mathbf{1}_{\ell}(p)p^{-a(G)s} + O\left(p^{-(2a(G)+1)s}\right)\right).
    \]
    It now suffices to consider the contribution of the prime $\ell$. It is natural to separate the proof in to two cases: $\ell = 2$ and $\ell \ne 2$.

    Suppose first that $\ell = 2$, so that $\Q(\zeta_2) = \Q(\zeta_2+\zeta_2^{-1}) = \Q$. Then
    \begin{align*}
        \zeta(a(G)s)^{-\frac{|G_2|}{\phi(2)}} = H(s) \prod_{p\nmid |G|}\left(1 - |G_2|p^{-a(G)s} + {|G_2|\choose 2}p^{-2a(G)s} + O\left(p^{-(2a(G)+1)s}\right)\right).
    \end{align*}
    Multiplying Euler factors implies
    \begin{align*}
        &D_{\Q}^{\text{\'et}}(G;s)\left(\prod_{d \ge \ell=2} \zeta_{\Q(\zeta_{d})}\left(|G|(1-1/d)s\right)^{-\frac{|G_d|}{\phi(d)}}\right)\\
        &\hspace{1cm}= H(s)\prod_{p\nmid |G|}\left(1 + \left({|G_2|\choose 2} - |G_2|^2\right)p^{-2a(G)s} + O\left(p^{-(2a(G)+1)s}\right)\right).
    \end{align*}
    Given that ${n\choose 2} - n^2 = -\frac{1}{2}n^2 - \frac{1}{2}n$, we conclude that
    \begin{align*}
        &D_{\Q}^{\text{\'et}}(G;s)\left(\prod_{d \ge \ell=2} \zeta_{\Q(\zeta_{d})}\left(|G|(1-1/d)s\right)^{-\frac{|G_d|}{\phi(d)}}\right)\zeta(2a(G)s)^{\frac{1}{2}|G_2|^2 + \frac{1}{2}|G_2|}\\
        &\hspace{1cm}= H(s)\prod_{p\nmid |G|}\left(1 + O\left(p^{-(2a(G)+1)s}\right)\right).
    \end{align*}
    We conclude the proof in this case by noting that
    \[
        -\frac{|G_\ell|}{2} - \frac{2\ell-3}{2\ell-2}|G_\ell|^2  + \frac{\ell-2}{2}|G_\ell| = -\frac{1}{2}|G_2|^2 - \frac{1}{2}|G_2|
    \]
    when $\ell = 2$.

    Now suppose $\ell > 2$. The proof is similar, except in this case the Euler factors of $\zeta_{\Q(\zeta_{\ell})}$ are slightly more complicated.
    \begin{align*}
        &\zeta_{\Q(\zeta_\ell)}(a(G)s)^{-\frac{|G_\ell|}{\phi(\ell)}}\\
        &=H_\ell(s) \prod_{p\nmid |G|}\Bigg(1 - \phi(\ell)\mathbf{1}_{1,\ell}(p)p^{-a(G)s} + {\phi(\ell)\choose 2}\mathbf{1}_{1,\ell}(p)p^{-2a(G)s} - \frac{\phi(\ell)}{2}\mathbf{1}_{-1,\ell}(p)p^{-2a(G)s}\\
        &\hspace{2.5cm}+ O\left(p^{-(2a(G)+1)s}\right) \Bigg)^{\frac{|G_\ell|}{\phi(\ell)}}\\
        &= H_\ell(s) \prod_{p\nmid |G|}\Bigg(1 - |G_\ell|\mathbf{1}_{1,\ell}(p)p^{-a(G)s} + \left({|G_\ell|/\phi(\ell)\choose 2}\phi(\ell)+{\phi(\ell)\choose 2}\frac{|G_\ell|}{\phi(\ell)}\right)\mathbf{1}_{1,\ell}(p)p^{-2a(G)s}\\
        &\hspace{2.5cm} - \frac{|G_\ell|}{2}\mathbf{1}_{-1,\ell}(p)p^{-2a(G)s} + O\left(p^{-(2a(G)+1)s}\right)\Bigg).
    \end{align*}
    Multiplying Euler factors implies
    \begin{align*}
        &D_{\Q}^{\text{\'et}}(G;s)\left(\prod_{d \ge \ell} \zeta_{\Q(\zeta_{d})}\left(|G|(1-1/d)s\right)^{-\frac{|G_d|}{\phi(d)}}\right)\\
        &= H(s)\prod_{p\nmid |G|}\Bigg(1 + \left({|G_\ell|/\phi(\ell)\choose 2}\phi(\ell) + {\phi(\ell)\choose 2}\frac{|G_\ell|}{\phi(\ell)} - |G_\ell|^2\right)\mathbf{1}_{1,\ell}(p)p^{-2a(G)s}\\
        &\hspace{2.5cm} - \frac{|G_\ell|}{2}\mathbf{1}_{-1,\ell}(p)+ O\left(p^{-(2a(G)+1)s}\right)\Bigg).
    \end{align*}
    Given that $\mathbf{1}_{1,\ell}(p) + \mathbf{1}_{-1,\ell}(p)$ is the characteristic function of the primes totally split in $\Q(\zeta_\ell+\zeta_\ell^{-1})$, we conclude that
    \begin{align*}
        &D_{\Q}^{\text{\'et}}(G;s)\left(\prod_{d \ge \ell=2} \zeta_{\Q(\zeta_{d})}\left(|G|(1-1/d)s\right)^{-\frac{|G_d|}{\phi(d)}}\right)\zeta_{\Q(\zeta_\ell+\zeta_{\ell}^{-1})}(2a(G)s)^{\frac{|G_\ell|}{2}}\\
        &= H(s)\prod_{p\nmid |G|}\Bigg(1 + \left({|G_\ell|/\phi(\ell)\choose 2}\phi(\ell) + {\phi(\ell)\choose 2}\frac{|G_\ell|}{\phi(\ell)} - |G_\ell|^2 + \frac{|G_\ell|}{2}\right)\mathbf{1}_{1,\ell}(p)p^{-2a(G)s}\\
        &\hspace{2.5cm} + O\left(p^{-(2a(G)+1)s}\right)\Bigg).
    \end{align*}
    Finally, multiplying in one more factor implies
    \begin{align*}
        &D_{\Q}^{\text{\'et}}(G;s)\left(\prod_{d \ge \ell=2} \zeta_{\Q(\zeta_{d})}\left(|G|(1-1/d)s\right)^{-\frac{|G_d|}{\phi(d)}}\right)\zeta_{\Q(\zeta_\ell+\zeta_{\ell}^{-1})}(2a(G)s)^{\frac{|G_\ell|}{2}}\\
        &\cdot \zeta_{\Q(\zeta_\ell)}(2a(G)s)^{-\left({|G_\ell|/\phi(\ell)\choose 2}\phi(\ell) + {\phi(\ell)\choose 2}\frac{|G_\ell|}{\phi(\ell)} - |G_\ell|^2 + \frac{|G_\ell|}{2}\right)}\\
        &= H(s)\prod_{p\nmid |G|}\left(1 + O\left(p^{-(2a(G)+1)s}\right)\right).
    \end{align*}
    The result then follows after noting that
    \[
        {|G_\ell|/\phi(\ell)\choose 2}\phi(\ell) + {\phi(\ell)\choose 2}\frac{|G_\ell|}{\phi(\ell)} - |G_\ell|^2 + \frac{|G_\ell|}{2} = -\frac{2\ell-3}{2\ell-2}|G_\ell|^2 + \frac{\ell-2}{2}|G_\ell|.
    \]
\end{proof}

\subsection{Moments of Dedekind Zeta Functions}

By H\"older's inequality, we can bound the integral moment
\[
    \int_0^T |D_\Q^{\text{\'et}}(G;\sigma+it)|dt
\]
solely in terms of the integral moments of Dedekind zeta functions
\[
    I_m(\sigma,\zeta_K;T) := \int_0^T |\zeta_K(\sigma+it)|^{2m}dt.
\]
We cite a few useful results bounding $I_m(\sigma,\zeta_K)$. The first is due to Chandrasekharan and Narasimhan \cite{chandrasekharan1963approximate}, following from their construction of the approximate functional equation for general $L$-functions. We note that, while we only require upper bounds, they prove an asymptotic in many cases.

\begin{lemma}[{\cite[Theorem 4]{chandrasekharan1963approximate}}]\label{lem:CN}
    For each number field $K/\Q$ and each integer $m \ge 0$,
    \[
        I_m(\sigma,\zeta_K;T) := \int_0^T |\zeta_K(\sigma+it)|^{2m} dt \ll \begin{cases}
        T & \sigma > 1 - \frac{1}{m[K:\Q]}\\
        T (\log T)^{m^2[K:\Q]} & \sigma = 1 - \frac{1}{m[K:\Q]}.
        \end{cases}
    \]
\end{lemma}

Strictly speaking, \cite[Theorem 4]{chandrasekharan1963approximate} is only stated for the case that $m=1$. The other integer values are proven similarly, as Chandresekharan and Narasimhan remark that an analogous approximate functional equation to \cite[Equation (65)]{chandrasekharan1963approximate} holds for $\zeta_K(s)^m$. With the approximate functional equation in hand, the rest of the proof proceeds similarly to the $m=1$ case.

It will also be useful to have sharp bounds at the edge of the critical strip. This is proven in large generality by Balasubramanian, Ivi\'c, and Ramachandra \cite{balasubramanian-ivic-ramachandra1993oneline}. Once again, we note that while we only require an upper bound they prove an asymptotic.

\begin{lemma}[{\cite[Theorem 3]{balasubramanian-ivic-ramachandra1993oneline}}]\label{lem:BIR}
    For each number field $K/\Q$, each complex number $z$, and each $\sigma \ge 1$
    \[
        I_z(\sigma,\zeta_K;T) \ll T.
    \]
\end{lemma}
The case $\sigma > 1$ follows from absolute convergence, while the boundary case $\sigma = 1$ follows from \cite[Theorem 3]{balasubramanian-ivic-ramachandra1993oneline} after noting that Brauer's induction theorem implies any Artin $L$-function can be expressed as a product of complex powers (really just integer powers) of Hecke $L$-functions, i.e.~functions in the set $S_1$ defined by Balasubramanian, Ivi\'c, and Ramachandra.

Finally, at the very boundary of what we can prove we are able to achieve slightly more by incorporating some pointwise bounds at the edge of the critical strip. We specifically apply the following result appearing in Tenenbaum's book.

\begin{lemma}[{\cite[Part II Theorem 3.22]{tenenbaum2015introduction}}]\label{lem:Ten}
    There exists a positive constant $c$ such that, for $|t|\ge 3$ and $\sigma \ge 1 - c/\log|t|$, we have
    \[
        |\zeta(\sigma+it)|^{-1} \ll \log |t|.
    \]
\end{lemma}

A similar result is likely true for Dedekind zeta functions by the same argument, although the author was unable to find a reference. We will only make use of this bound for the Riemann zeta function in our examples.

\subsection{Proofs of square root savings}

We first consider the case $G = C_3$. In this case, it suffices to count $C_3$-\'etale algebras, as all but one $C_3$-\'etale algebra is in fact a field. This allows us to avoid using a sieve at the end of the argument.

\begin{proof}[Proof of \Cref{cor:uncond_square_root_savings} for $G = C_3$]
    Following \Cref{prop:abelian_meromorphic_continuation} and noting that $a(C_3) = 2$, the generating series for $C_3$-\'etale algebras is given by
    \[
        D_\Q^{\text{\'et}}(C_3;s) = \zeta_{\Q(\zeta_3)}(2s)\zeta(4s)^{-1} \zeta_{\Q(\zeta_3)}(4s)^{-2}H(s),
    \]
    where $H(s)$ is an Euler product that converges absolutely on the region ${\rm Re}(s) > \frac{1}{2a(C_p)+1} = \frac{1}{5}$. In particular, $|H(s)|$ is uniformly bounded on the region $\sigma \ge \frac{1}{2a(C_p)} = \frac{1}{4}$.
    
    H\"older's inequality now implies
    \begin{align*}
        \int_0^T |D_\Q^{\text{\'et}}(C_3;\sigma + it)| dt \ll& \left(\int_0^T |\zeta_{\Q(\zeta_3)}(2(\sigma+it))|^2 dt\right)^{1/2}\left(\int_0^T |\zeta(4(\sigma+it))|^{-4} dt\right)^{1/4}\\
        &\cdot\left(\int_0^T |\zeta_{\Q(\zeta_3)}(4(\sigma+it))|^{-8} dt\right)^{1/4}.
    \end{align*}
    for each $\sigma \ge \frac{1}{4}$.

    In the range $\sigma \ge \frac{1}{4}$, \Cref{lem:CN} implies
    \begin{align*}
        \int_0^T |\zeta_{\Q(\zeta_p)}(2(\sigma+it))|^2 dt &= \frac{1}{2}I_1(2\sigma,\zeta_{\Q(\zeta_3)};2T) \ll T(\log T)^{2}.
    \end{align*}
    Next, for any $\sigma \ge \frac{1}{4}$, \Cref{lem:BIR} implies that
    \begin{align*}
        \int_0^T |\zeta(4(\sigma+it))|^{-4} dt &\ll I_{-2}(4\sigma,\zeta;4T) \ll T\\
        \int_0^T |\zeta_{\Q(\zeta_3)}(4(\sigma+it))|^{-8} dt &\ll I_{-4}(4\sigma,\zeta_{\Q(\zeta_3)};4T) \ll T.
    \end{align*}
    We have now shown that
    \begin{align*}
        \int_0^T |D_\Q^{\text{\'et}}(C_3;\sigma + it)| dt \ll& \left(T (\log T)^{2}\right)^{1/2}T^{1/4}T^{1/4} = T\log T.
    \end{align*}
    for each $\sigma\ge \frac{1}{4}$.
    
    The result then follows from \Cref{thm:inexplicit_main} by taking $\sigma_a = \frac{1}{a(C_3)} = \frac{1}{2}$, $\sigma_a - \delta = \frac{1}{4}$, $\widetilde{\eta}=\eta=1$, and $\beta = 1$.
\end{proof}

In fact, the argument above can be adapted to prove that
\[
    \#\mathcal{F}_\Q(C_p;X) = c_{p-1}(C_p)X^{\frac{1}{p-1}} + O\left(X^{\frac{p-2}{(p-1)^2}}(\log X)^{\frac{p+1}{2}}\right)
\]
for each odd prime $p$. However, for $p > 3$ this is strictly weaker than \cite[Corollary 1.5]{alberts2024power}, which implies
\[
    \#\mathcal{F}_\Q(C_p;X) = c_{p-1}(C_p)X^{\frac{1}{p-1}} + O\left(X^{\frac{p+2}{(p-1)(p+5)}+\epsilon}\right).
\]

The proof of \Cref{cor:uncond_square_root_savings} for the even cyclic groups is similar to that of $C_3$, with essentially two differences: firstly, we need to sieve from \'etale algebras to number fields at the end of the argument. Second, the groups $C_{16}$ and $C_{2p}$ for $p\ge 5$ a prime require the pointwise bound in \Cref{lem:Ten} instead of \Cref{lem:BIR}.

When $G=C_{2n}$ is an even cyclic group, we remark that $|G_d| = \phi(d)$ for each $d\mid |G|$, and in particular $-\frac{1}{2}|G_2|^2 - \frac{1}{2}|G_2| = -1$. It then follows from \Cref{prop:abelian_meromorphic_continuation} that
\[
    D_{\Q}^{\text{\'et}}(C_{2n};s) = \left(\prod_{\substack{d\mid 2n\\d > 1}} \zeta_{\Q(\zeta_d)}(2n(1-1/d)s)\right) \zeta(2ns)^{-1} H(s),
\]
where $a(C_{2n}) = n$ and $H(s)$ is an Euler product that converges absolutely on the region ${\rm Re}(s) > \frac{1}{2n+1}$, and so in particular is bounded uniformly in the region ${\rm Re}(s) \ge \frac{1}{2n}$.

\begin{proof}[Proof of \Cref{cor:uncond_square_root_savings} for $G=C_4$]
    H\"older's inequality implies
    \begin{align*}
        \int_0^T |D_\Q^{\text{\'et}}(C_4;\sigma + it)| dt \ll& \left(\int_0^T |\zeta(2(\sigma+it))|^4dt\right)^{1/4}\left(\int_0^T |\zeta_{\Q(i)}(3(\sigma+it))|^4dt\right)^{1/4}\\
        &\cdot\left(\int_0^T |\zeta(4(\sigma+it))|^{-2}dt\right)^{1/2}.
    \end{align*}
    For each $\sigma \ge 1/4$, we bound the first two factors using \Cref{lem:CN} and the third factor using \Cref{lem:BIR} to prove that
    \begin{align*}
        \int_0^T |D_\Q^{\text{\'et}}(C_4;\sigma + it)| dt \ll& \left(T(\log T)^4\right)^{1/4}\left(T(\log T)^8\right)^{1/4}\left(T\right)^{1/2} = T(\log T)^3.
    \end{align*}
    Applying \Cref{thm:inexplicit_main} with $\sigma_a = \frac{1}{2}$, $\sigma_a - \delta = \frac{1}{4}$, $\widetilde{\eta}=\eta=1$, and $\beta = 3$ implies the number of $C_4$-\'etale algebras is given by
    \[
        \#\mathcal{F}_{\Q}(C_4;X) + \#\mathcal{F}_\Q(C_2;X^{1/2}) + 1= c_0X^{1/2} + c_1X^{1/3} + O(X^{1/4}(\log X)^4).
    \]
    The result then follows after incorporating the classical power savings
    \[
        \#\mathcal{F}_\Q(C_2;X^{1/2}) = c_2 X^{1/2} + O(X^{1/4}).
    \]
\end{proof}

\begin{proof}[Proof of \Cref{cor:uncond_square_root_savings} for $G=C_6$]
    H\"older's inequality implies
    \begin{align*}
        \int_0^T |D_\Q^{\text{\'et}}(C_6;\sigma + it)| dt \ll& \left(\int_0^T |\zeta(3(\sigma+it))|^4dt\right)^{1/4}\left(\int_0^T |\zeta_{\Q(\zeta_3)}(4(\sigma+it))|^2dt\right)^{1/2}\\
        &\cdot\left(\int_0^T |\zeta_{\Q(\zeta_6)}(5(\sigma+it))|^6 dt\right)^{1/6}\left(\int_0^T |\zeta(6(\sigma+it))|^{-12}dt\right)^{1/12}.
    \end{align*}
    For each $\sigma \ge 1/6$, we bound the first three factors using \Cref{lem:CN} and the fourth factor using \Cref{lem:BIR} to prove that
    \begin{align*}
        \int_0^T |D_\Q^{\text{\'et}}(C_p;\sigma + it)| dt \ll& \left(T(\log T)^4\right)^{1/4}T^{1/2}\left(T(\log T)^{18}\right)^{1/6}T^{1/12} = T(\log T)^4.
    \end{align*}
    Applying \Cref{thm:inexplicit_main} with $\sigma_a = \frac{1}{3}$, $\sigma_a - \delta = \frac{1}{6}$, $\widetilde{\eta}=\eta=1$, and $\beta = 4$ implies the number of $C_6$-\'etale algebras is given by
    \begin{align*}
        &\#\mathcal{F}_{\Q}(C_6;X) + \#\mathcal{F}_\Q(C_3;X^{1/2}) + \#\mathcal{F}_\Q(C_2;X^{1/3}) + 1\\
        &= c_0 X^{1/3} + c_1 X^{1/4} + c_2 X^{1/5} + O(X^{1/6}(\log X)^5).
    \end{align*}
    The result then follows by the power savings we previous proved for $C_3$-extensions and the classical power savings for $C_2$-extensions
    \begin{align*}
        \#\mathcal{F}_\Q(C_3;X^{1/2}) &= c_3 X^{1/4} + O(X^{1/8}(\log X)^2)\\
        \#\mathcal{F}_\Q(C_2;X^{1/3}) &= c_4 X^{1/3} + O(X^{1/6}).
    \end{align*}
\end{proof}

\begin{proof}[Proof of \Cref{cor:uncond_square_root_savings} for $G=C_8$]
    H\"older's inequality implies
    \begin{align*}
        \int_0^T |D_\Q^{\text{\'et}}(C_8;\sigma + it)| dt \ll& \left(\int_0^T |\zeta(4(\sigma+it))|^4dt\right)^{1/4}\left(\int_0^T |\zeta_{\Q(i)}(6(\sigma+it))|^4dt\right)^{1/4}\\
        & \cdot \left(\int_0^T |\zeta_{\Q(\zeta_8)}(7(\sigma+it))|^4dt\right)^{1/4}\left(\int_0^T |\zeta(8(\sigma+it))|^{-4}dt\right)^{1/4}.
    \end{align*}
    For each $\sigma \ge 1/8$, we bound the first three factors using \Cref{lem:CN} and the fourth factor using \Cref{lem:BIR} to prove that
    \begin{align*}
        \int_0^T |D_\Q^{\text{\'et}}(C_p;\sigma + it)| dt \ll& \left(T(\log T)^4\right)^{1/4}\left(T(\log T)^8\right)^{1/4}\left(T(\log T)^{16}\right)^{1/4}\left(T\right)^{1/2} = T(\log T)^{7}.
    \end{align*}
    Applying \Cref{thm:inexplicit_main} with $\sigma_a = \frac{1}{4}$, $\sigma_a - \delta = \frac{1}{8}$, $\widetilde{\eta}=\eta=1$, and $\beta = 7$ implies the number of $C_8$-\'etale algebras is given by
    \begin{align*}
        &\#\mathcal{F}_{\Q}(C_8;X) + \#\mathcal{F}_{\Q}(C_4;X^{1/2}) + \#\mathcal{F}_\Q(C_2;X^{1/4}) + 1\\
        &= c_0X^{1/4} + c_1X^{1/6} + c_2X^{1/7} + O(X^{1/8}(\log X)^8).
    \end{align*}
    The result then follows by our previous proof of \Cref{cor:uncond_square_root_savings} for $G=C_4$ and the classical power savings for $C_2$-extensions showing that
    \begin{align*}
        \#\mathcal{F}_\Q(C_4;X^{1/2}) &= c_3 X^{1/4} + c_4 X^{1/6} + O(X^{1/8}(\log X)^4)\\
        \#\mathcal{F}_\Q(C_2;X^{1/4}) &= c_5 X^{1/4} + O(X^{1/8}).
    \end{align*}
\end{proof}

\begin{proof}[Proof of \Cref{cor:uncond_square_root_savings} for $G=C_{16}$]
    The exact process we used for $G=C_4$, $C_6$, and $C_8$ just barely does not work here. It is not possible to include the average value of $\zeta(2a(C_{16})s)^{-1}$, and still have a small enough average for the other terms to apply \Cref{lem:CN}. Instead, we apply \Cref{lem:Ten} to bound this factor above by $\log T$, and then use H\"older's inequality to bound
    \begin{align*}
        \int_0^T |D_\Q^{\text{\'et}}(C_{16};\sigma + it)| dt \ll& \left(\int_0^T |\zeta(8(\sigma+it))|^4dt\right)^{1/4}\left(\int_0^T |\zeta_{\Q(i)}(12(\sigma+it))|^4dt\right)^{1/4}\\
        & \cdot \left(\int_0^T |\zeta_{\Q(\zeta_8)}(14(\sigma+it))|^4dt\right)^{1/4}\left(\int_0^T |\zeta_{\Q(\zeta_{16})}(15(\sigma+it))|^{-4}dt\right)^{1/4}\\
        &\cdot \log T.
    \end{align*}
    For each $\sigma \ge 1/16$, we bound the first four factors using \Cref{lem:CN} to prove that
    \begin{align*}
        \int_0^T |D_\Q^{\text{\'et}}(C_{16};\sigma + it)| dt &\ll \left(T(\log T)^4\right)^{1/4}\left(T(\log T)^8\right)^{1/4}\left(T(\log T)^{16}\right)^{1/4}\left(T(\log T)^{32}\right)^{1/4}\log T\\
        &= T(\log T)^{16}.
    \end{align*}
    Applying \Cref{thm:inexplicit_main} with $\sigma_a = \frac{1}{8}$, $\sigma_a - \delta = \frac{1}{16}$, $\widetilde{\eta}=\eta=1$, and $\beta = 16$ implies the number of $C_8$-\'etale algebras is given by
    \begin{align*}
        &\#\mathcal{F}_{\Q}(C_{16};X) + \#\mathcal{F}_{\Q}(C_8;X^{1/2}) + \#\mathcal{F}_\Q(C_4;X^{1/4}) + \#\mathcal{F}_\Q(C_2;X^{1/8}) + 1\\
        &= c_0X^{1/8} + c_1X^{1/12} + c_2X^{1/14} + c_3 X^{1/15} + O(X^{1/16}(\log X)^{17}).
    \end{align*}
    The result then follows by our previous proofs of \Cref{cor:uncond_square_root_savings} for $G=C_4,C_8$ and the classical power savings for $G=C_2$, which together show that
    \begin{align*}
        \#\mathcal{F}_\Q(C_8;X^{1/2}) &= c_4 X^{1/8} + c_5 X^{1/12} + c_6 X^{1/14} + O(X^{1/16}(\log X)^8)\\
        \#\mathcal{F}_\Q(C_4;X^{1/4}) &= c_7 X^{1/8} + c_8 X^{1/12} + O(X^{1/16}(\log X)^4)\\
        \#\mathcal{F}_\Q(C_2;X^{1/8}) &= c_9 X^{1/8} + O(X^{1/16}).
    \end{align*}
\end{proof}

\begin{proof}[Proof of \Cref{cor:uncond_square_root_savings} for $G=C_{2p}$ with $p\ge 5$]
    Similar to the case $G=C_{16}$, it is not possible to include the average value of $\zeta(2a(C_{p})s)^{-1}$ and still have a small enough average for the other terms to apply \Cref{lem:CN}. Instead, we apply \Cref{lem:Ten} to bound this factor above by $\log T$, and then use H\"older's inequality to bound
    \begin{align*}
        \int_0^T |D_\Q^{\text{\'et}}(C_{2p};\sigma + it)| dt \ll& \left(\int_0^T |\zeta(p(\sigma+it))|^4dt\right)^{1/4}\left(\int_0^T |\zeta_{\Q(\zeta_p)}((2p-2)(\sigma+it))|^2dt\right)^{1/2}\\
        & \cdot \left(\int_0^T |\zeta_{\Q(\zeta_{2p})}((2p-1)(\sigma+it))|^4dt\right)^{1/4}\log T.
    \end{align*}
    For each $\sigma \ge 1/2p$, we bound the first three factors using \Cref{lem:CN} to prove that
    \begin{align*}
        \int_0^T |D_\Q^{\text{\'et}}(C_p;\sigma + it)| dt \ll& \left(T(\log T)^4\right)^{1/4}\left(T\right)^{1/2}\left(T\right)^{1/4}\log T = T(\log T)^{2}.
    \end{align*}
    Applying \Cref{thm:inexplicit_main} with $\sigma_a = \frac{1}{p}$, $\sigma_a - \delta = \frac{1}{2p}$, $\widetilde{\eta}=\eta=1$, and $\beta = 2$ implies the number of $C_{2p}$-\'etale algebras is given by
    \begin{align*}
        &\#\mathcal{F}_{\Q}(C_{2p};X) + \#\mathcal{F}_{\Q}(C_p;X^{1/2}) + \#\mathcal{F}_\Q(C_2;X^{1/p}) + 1\\
        &= c_0X^{\frac{1}{p}} + c_1X^{\frac{1}{2p-2}} + c_2X^{\frac{1}{2p-1}} + O\left(X^{\frac{1}{2p}}(\log X)^{3}\right).
    \end{align*}
    The result then follows by the power savings for $G=C_p$ proven in \cite[Corollary 1.5]{alberts2024power} and the classical power savings for $G=C_2$, which together show that
    \begin{align*}
        \#\mathcal{F}_\Q(C_p;X^{1/2}) &= c_3 X^{\frac{1}{2p-2}} + O\left(X^{\frac{p+2}{2(p-1)(p+5)}+\epsilon}\right)\\
        \#\mathcal{F}_\Q(C_2;X^{1/p}) &= c_4 X^{\frac{1}{p}} + O\left(X^{\frac{1}{2p}}\right).
    \end{align*}
    The result follows by choosing $\epsilon$ sufficiently small, after noting that $\frac{p+2}{2(p-1)(p+5)} < \frac{1}{2p}$.
\end{proof}

\subsection{Remarks on the constants $c_d(C_n)$}\label{subsec:leading_constants}

When counting \'etale algebras, the coefficients of each term in the asymptotic expansion are certainly given by special values of Dedekind zeta functions times convergent Euler products, as the generating Dirichlet series $D_\Q^{\text{\'et}}(G;s)$ is a product of Dedekind zeta functions times a convergent Euler product in a neighborhood of each pole. Thus, the coefficients $c_d(C_n)$ are finite sums of these types of constants as claimed in \Cref{cor:uncond_square_root_savings}.

The more interesting question of ``do these constants vanish'' was partially answered in \cite{alberts2024power}. It turns out that each term appearing in \Cref{cor:uncond_square_root_savings} is covered by one of the cases in \cite[Theorem 1.7]{alberts2024power}, as the groups considered are not very complicated. In particular, if $n=3,4,8,16$, then each $g\in C_n$ is either a generator or belongs to the Frattini subgroup. This implies that each index satisfies the conditions of \cite[Theorem 1.7(ii) or (iv)]{alberts2024power}. If $n = 2p$ for $p\ge 3$, then $g\in C_{2p} = C_2\times C_p$ satisfies one of the following:
\begin{enumerate}[(a)]
    \item $g$ is a generator of $C_{2p}$, or
    \item $\{h\in C_{2p} : \ind(h) < \ind(g)\}\subseteq \Phi(C_{2p})\cup C_{2p}[2]$.
\end{enumerate}
This are exactly the conditions for \cite[Theorem 1.7(iii,iv)]{alberts2024power}. By appealing to this result, we conclude the following:

Let $c_{n(1-1/d)}(C_n)$ be one of the constants appearing in \Cref{cor:uncond_square_root_savings} corresponding to a divisor $d\mid n$. Then $c_{n(1-1/d)}(C_n) \ne 0$ if and only if
\[
    \prod_{\substack{m\mid n\\ 1 < m < d}}\zeta_{\Q(\zeta_m)}\left(\frac{1-1/m}{1-1/d}\right) \ne 0.
\]
In case $d=\ell$ is the smallest prime dividing $n$, this is the empty product which always equals $1$, re-verifying that the leading term does not vanish.

By tracing through the surjective sieve, one can write down an explicit formula for these constants in terms of local \'etale algebras. It would be interesting to write a completely explicit formulation, which would require a detailed description of the discriminants of local \'etale algebras at wild places.

For individual small groups, we can refer to the LMFDB \cite{lmfdb} to compute explicit constants. To demonstrate the approximate shape of such results, we do this for $G=C_4$.

\begin{proposition}
    Let $c_d(C_n)$ denote the constants appearing in \Cref{cor:uncond_square_root_savings}. Then
    \begin{align*}
        c_2(C_4) &=\frac{28+\sqrt{2}}{24\zeta(2)}\prod_{p\equiv 1\pmod 4}\left(\frac{1 + 2p^{-3/2} - p^{-2} - 2p^{-5/2}}{1-p^{-2}}\right) -\frac{1}{\zeta(2)}\\
        c_3(C_4) &= \zeta(2/3)L(1,\chi_4)\left(\frac{3(4-2^{2/3})}{16}\right)\prod_{p\ne 2}\left(1 -p^{-4/3} - p^{-2} + p^{-10/3}\right)\\
        &\hspace{0.5cm} \cdot \prod_{p\equiv 1\pmod 4}\left(\frac{1-p^{-4/3}-2p^{-5/3}-3p^{-2}+2p^{-7/3}+4p^{-8/3}+2p^{-3}-p^{-10/3}-2p^{-11/3}}{1-p^{-4/3}-p^{-2}+p^{-10/3}}\right)
    \end{align*}
\end{proposition}

\begin{proof}
    Using the complete list of $C_2$- and $C_4$-extensions of $\Q_2$ on \cite{lmfdb}, we can explicitly compute
    \[
        D_\Q^{\text{\'et}}(C_4;s) = \left(1 + 2^{-2s} + 2\cdot 2^{-6s} + 4\cdot 2^{-11s}\right) \prod_{p\ne 2}\left(1 + p^{-2s} + (1+\chi_4(p))p^{-3s}\right),
    \]
    where $\chi_4(p) = (-1/p)$ is the quadratic character associated to $\Q(i)$. We can use this expression to directly calculate the residues of $D_\Q^{\text{\'et}}(C_4;s)X^s/s$ at the poles $s=1/2$ and $s=1/3$.
    
    Factoring out $\zeta(2s)$ gives the explicit expansion
    \begin{align*}
        D_\Q^{\text{\'et}}(C_4;s) &= \zeta(2s)\left(1 + 2^{-2s} + 2\cdot 2^{-6s} + 4\cdot 2^{-11s}\right)\left(1-2^{-2s}\right)\\
        &\hspace{0.5cm}\cdot \prod_{p\ne 2}\left(1 + p^{-2s} + (1+\chi_4(p))p^{-3s}\right)\left(1-p^{-2s}\right)\\
        &= \zeta(2s)\left(1 - 2^{-4s} + 2\cdot 2^{-6s} - 2\cdot 2^{-8s} + 4\cdot 2^{-11s} - 4\cdot 2^{-13s}\right)\\
        &\hspace{0.5cm}\cdot \prod_{p\ne 2}\left(1 + (1+\chi_4(p))p^{-3s} - p^{-4s} - (1+\chi_4(p))p^{-5s}\right)\\
    \end{align*}
    The Euler product converges absolutely at $s=1/2$, so we find that
    \begin{align*}
        c_2(C_4) + c_1(C_2) &=\underset{s=1/2}{\rm Res}\left(D_\Q^{\text{\'et}}(C_4;s)\frac{X^s}{s}\right)\\
        &= \frac{28+\sqrt{2}}{24\zeta(2)}\prod_{p\equiv 1\pmod 4}\left(\frac{1 + 2p^{-3/2} - p^{-2} - 2p^{-5/2}}{1-p^{-2}}\right),
    \end{align*}
    from which we conclude
    \begin{align*}
        c_2(C_4) &=\frac{28+\sqrt{2}}{24\zeta(2)}\prod_{p\equiv 1\pmod 4}\left(\frac{1 + 2p^{-3/2} - p^{-2} - 2p^{-5/2}}{1-p^{-2}}\right) - \frac{1}{\zeta(2)}.
    \end{align*}

    Further factoring out $\zeta_{\Q(i)}(3s)$ gives the expression
    \begin{align*}
        &D_\Q^{\text{\'et}}(C_4;s)\\
        &= \zeta(2s)\zeta_{\Q(i)}(3s)\left(1 - 2^{-4s} + 2\cdot 2^{-6s} - 2\cdot 2^{-8s} + 4\cdot 2^{-11s} - 4\cdot 2^{-13s}\right)\left(1-2^{-3s}\right)\\
        &\hspace{0.5cm}\cdot \prod_{p\ne 2}\left(1 + (1+\chi_4(p))p^{-3s} - p^{-4s} - (1+\chi_4(p))p^{-5s}\right)\left(1 - (1+\chi_4(p))p^{-3s} + \chi_4(p)p^{-6s}\right)\\
        &= \zeta(2s)\zeta_{\Q(i)}(3s)\Big(1-2^{-3s}-2^{-4s}+2\cdot2^{-6s} + 2^{-7s} - 2\cdot 2^{-8s} - 2\cdot 2^{-9s}\\
        &\hspace{3cm}+ 6\cdot 2^{-11s} - 4\cdot 2^{-13s} - 4\cdot 2^{-14s} + 2\cdot 2^{-16s}\Big)\\
        &\hspace{0.5cm}\cdot \prod_{p\ne 2}\Big(1 - p^{-4s} - (1+\chi_4(p))p^{-5s} + \chi_4(p)p^{-6s} - 2(1+\chi_4(p))p^{-6s} + (1+\chi_4(p))p^{-7s}\\
        &\hspace{2cm}+ 2(1+\chi_4(p))p^{-8s} + (1+\chi_4(p))p^{-9s} - \chi_4(p)p^{-10s} - (1+\chi_4(p))p^{-11s}\Big)
    \end{align*}
    The Euler product converges absolutely at $s=1/3$, so we find that
    \begin{align*}
        &c_3(C_4)=\underset{s=1/2}{\rm Res}\left(D_\Q^{\text{\'et}}(C_4;s)\frac{X^s}{s}\right)\\
        &= \zeta(2/3)L(1,\chi_4)\left(\frac{3(4-2^{2/3})}{16}\right)\prod_{p\ne 2}\left(1 -p^{-4/3} - p^{-2} + p^{-10/3}\right)\\
        &\hspace{0.5cm} \cdot \prod_{p\equiv 1\pmod 4}\left(\frac{1-p^{-4/3}-2p^{-5/3}-3p^{-2}+2p^{-7/3}+4p^{-8/3}+2p^{-3}-p^{-10/3}-2p^{-11/3}}{1-p^{-4/3}-p^{-2}+p^{-10/3}}\right).
    \end{align*}
\end{proof}

\bibliographystyle{alpha}
\bibliography{main.bbl}

\begin{thebibliography}{{LMF}25}

\bibitem[Alb24a]{alberts2024explicit}
Brandon Alberts.
\newblock Explicit analytic continuation of euler products.
\newblock {\em arXiv preprint arXiv:2406.18190}, 2024.
\newblock Preprint available at \url{https://arxiv.org/abs/2406.18190}.

\bibitem[Alb24b]{alberts2024power}
Brandon Alberts.
\newblock Power savings for counting (twisted) abelian extensions of number
  fields, 2024.
\newblock Preprint available at \url{https://arxiv.org/abs/2402.03475}.

\bibitem[BIR93]{balasubramanian-ivic-ramachandra1993oneline}
R~Balasubramanian, A~Ivi{\'c}, and K~Ramachandra.
\newblock An application of the {H}ooley-{H}uxley contour.
\newblock {\em Acta Arithmetica}, 65(1):45--51, 1993.

\bibitem[CN62]{chandrasekharan1962functional}
K~Chandrasekharan and Raghavan Narasimhan.
\newblock Functional equations with multiple gamma factors and the average
  order of arithmetical functions.
\newblock {\em Annals of Mathematics}, 76(1):93--136, 1962.

\bibitem[CN63]{chandrasekharan1963approximate}
K~Chandrasekharan and Raghavan Narasimhan.
\newblock The approximate functional equation for a class of zeta-functions.
\newblock {\em Mathematische Annalen}, 152(1):30--64, 1963.

\bibitem[Con01]{conrey2001functions}
J~Brian Conrey.
\newblock L-functions and random matrices.
\newblock In {\em Mathematics unlimited—2001 and Beyond}, pages 331--352.
  Springer, 2001.

\bibitem[FLN18]{frei-loughran-newton2018}
Christopher Frei, Daniel Loughran, and Rachel Newton.
\newblock The {H}asse norm principle for abelian extensions.
\newblock {\em American Journal of Mathematics}, 140(6):1639–1685, 2018.

\bibitem[Har13]{harper2013sharp}
Adam~J Harper.
\newblock Sharp conditional bounds for moments of the riemann zeta function.
\newblock {\em arXiv preprint arXiv:1305.4618}, 2013.
\newblock Preprint available at \url{https://arxiv.org/abs/1305.4618}.

\bibitem[HRS19]{heap-radziwill-sound2019sharp}
Winston Heap, Maksym Radziwi{\l}{\l}, and Kannan Soundararajan.
\newblock Sharp upper bounds for fractional moments of the {R}iemann zeta
  function.
\newblock {\em The Quarterly Journal of Mathematics}, 70(4):1387--1396, 2019.

\bibitem[KS00]{keating-snaith2000random}
Jon~P Keating and Nina~C Snaith.
\newblock Random matrix theory and $\zeta$ (1/2+ it).
\newblock {\em Communications in Mathematical Physics}, 214(1):57--89, 2000.

\bibitem[Lan15]{landau1915}
Edmund Landau.
\newblock {\"U}ber die anzahl der gitterpunkte in gewissen bereichen.
\newblock {\em Nachrichten von der Gesellschaft der Wissenschaften zu
  G{\"o}ttingen, Mathematisch-Physikalische Klasse}, page 209–243, 1915.

\bibitem[LDTT22]{lowry-duda-taniguchi-thorne2022uniform}
David Lowry-Duda, Takashi Taniguchi, and Frank Thorne.
\newblock Uniform bounds for lattice point counting and partial sums of zeta
  functions.
\newblock {\em Mathematische Zeitschrift}, 300(3):2571--2590, 2022.

\bibitem[LLM25]{languasco-lunia-moree2025counting}
Alessandro Languasco, Rashi Lunia, and Pieter Moree.
\newblock Counting ideals in abelian number fields, 2025.
\newblock Preprint available at \url{https://arxiv.org/abs/2504.09469}.

\bibitem[{LMF}25]{lmfdb}
The {LMFDB Collaboration}.
\newblock The {L}-functions and modular forms database.
\newblock \url{https://www.lmfdb.org}, 2025.
\newblock [Online; accessed 30 July 2025].

\bibitem[M{\"a}k85]{maki1985}
Sirpa M{\"a}ki.
\newblock On the density of abelian number fields.
\newblock {\em Annales Academiae Scientiarum Fennicae. Series A. I,
  Mathematica. Dissertationes, 0355-0087}, 54, 1985.

\bibitem[PTBZ25]{pierce-turnage-butterbaugh-zaman2025guide}
Lillian~B Pierce, Caroline~L Turnage-Butterbaugh, and Asif Zaman.
\newblock A guide to tauberian theorems for arithmetic applications, 2025.
\newblock Preprint available at \url{https://arxiv.org/abs/2504.16233}.

\bibitem[Rou11]{roux2011}
Mathieu Roux.
\newblock Th{\'e}orie de l'information, s{\'e}ries de {D}irichlet, et analyse
  d'algorithmes.
\newblock {\em Th{\'e}orie de l'information [cs.IT]. {U}niversit{\'e} de
  {C}aen}, 2011.
\newblock Francais. ⟨NNT : ⟩. ⟨tel-01076421⟩.

\bibitem[Sou22]{soundararajan2022distribution}
Kannan Soundararajan.
\newblock The distribution of values of zeta and {L}-functions.
\newblock In {\em ICM—International Congress of Mathematicians}, volume~2,
  pages 1260--1310, 2022.
\newblock Preprint available at \url{https://arxiv.org/abs/2112.03389}.

\bibitem[Ten15]{tenenbaum2015introduction}
G{\'e}rald Tenenbaum.
\newblock {\em Introduction to analytic and probabilistic number theory},
  volume 163.
\newblock American Mathematical Soc., 2015.

\bibitem[Wri89]{wright1989}
David~J. Wright.
\newblock Distribution of discriminants of abelian extensions.
\newblock {\em Proceedings of the London Mathematical Society},
  s3-58(1):17--50, 1989.

\end{thebibliography}

\end{document}